\documentclass[11pt,a4paper,oneside]{article}
\usepackage[small, bf]{titlesec}
\usepackage[noblocks]{authblk}

\usepackage[english]{babel}
\usepackage{enumerate} 
\usepackage{amssymb,amsmath,amsfonts,amsthm}
\usepackage{colonequals}
\newcommand{\euler}{\mathrm{e}}
\newcommand{\BV}{\mathrm{BV}}
\newcommand{\drm}{\mathrm{d}}
\newcommand{\RR}{\mathbb{R}}
\newcommand{\Eins}{\mathbf{1}}
\newcommand{\CC}{\mathbb{C}}
\newcommand{\NN}{\mathbb{N}}
\newcommand{\ZZ}{\mathbb{Z}}
\newcommand{\EE}{\mathbb{E}}
\newcommand{\PP}{\mathbb{P}}
\newcommand{\cc}{\mathfrak{C}}
\newcommand{\rkl}[1]{\ensuremath{\left(#1\right)}}
\newcommand{\cd}[1]{\ensuremath{\mathfrak{C}_{#1}}}
\renewcommand{\epsilon}{\varepsilon}
\newcommand{\nn}{\nonumber}
\DeclareMathOperator{\supp}{\mathop{supp}}

\DeclareMathOperator{\dist}{\mathop{dist}}
\DeclareMathOperator{\Tr}{\mathop{Tr}}
\newtheorem{theorem}{Theorem}[section]
\newtheorem{lemma}[theorem]{Lemma}
\newtheorem{proposition}[theorem]{Proposition}
\theoremstyle{definition}
\newtheorem{definition}[theorem]{Definition}
\newtheorem{assumption}{Assumption}

\theoremstyle{remark}
\newtheorem{remark}[theorem]{Remark}
\newcommand{\myfootnote}[1]{%
\renewcommand{\thefootnote}{}%
\footnotetext{#1}%
\renewcommand{\thefootnote}{\arabic{footnote}}%
}
\usepackage[DIV=12]{typearea}
\usepackage{microtype}
\begin{document}
%
% %
% %%%%%%%%%%%%%%%%%%%%%%%%%%%%%%%%%%%%%%%%%%%%%%%%%%%%%%%%%%%%%%%%%%%%%%%%%%%%%%%%%%%
% %----------------------------------------------------------------------------------
% %         T I T L E
% %----------------------------------------------------------------------------------
% %%%%%%%%%%%%%%%%%%%%%%%%%%%%%%%%%%%%%%%%%%%%%%%%%%%%%%%%%%%%%%%%%%%%%%%%%%%%%%%%%%%
% %
%
%
\title{Wegner estimate and localization for alloy-type models with sign-changing exponentially decaying single-site potentials}

\author[1,3]{Karsten Leonhardt}
\author[2]{Norbert Peyerimhoff}
\author[1]{Martin Tautenhahn}
\author[1]{Ivan Veseli\'c}

\affil[1]{Faculty of Mathematics, Chemnitz University of Technology, 09107 Chemnitz, Germany}

\affil[2]{Department of Mathematical Sciences, Durham University, DH1 3LE, Great Britain} 

\affil[3]{Current address: Max Planck Institute for the Physics of Complex Systems, 01187 Dresden, Germany} 
\date{}
%-----------------------------------------
%
\maketitle

\begin{abstract}
We study Schr\"odinger operators on $L^2 (\RR^d)$ and $\ell^2(\ZZ^d)$ with a random potential of alloy-type. 
The single-site potential is
assumed to be exponentially decaying but not necessarily of fixed sign.
In the continuum setting we require a generalized step-function shape.
Wegner estimates are bounds on the average number of eigenvalues in an energy interval of finite box restrictions of these types of operators. 
In the described situation a Wegner estimate which is polynomial in the volume of the box and linear in the size of the energy interval holds. We apply the established Wegner estimate as an ingredient for a localization proof via multiscale analysis. 
%\myfootnote{Date: \today, \jobname.tex}
\myfootnote{Keywords: random Schr\"odinger operators, alloy-type model, discrete alloy-type model, integrated density of states, Wegner estimate, single-site potential}
\myfootnote{MSC2010: 82B44, 60H25, 35J10}
\end{abstract}
\tableofcontents
%
%% %%%%%%%%%%%%%%%%%%%%%%%%%%%%%%%%%%%%%%%%%%%%%%%%%%%%%%%%%%%%%%%%%%%%%%%%%%%%%%%%%%%
%--------------------------------------------------------------------------------
%
%          INTRO
% %--------------------------------------------------------------------------------
%%
%%%%%%%%%%%%%%%%%%%%%%%%%%%%%%%%%%%%%%%%%%%%%%%%%%%%%%%%%%%%%%%%%%%%%%%%%%%%%%%%%
%%
%
\section{Introduction}
\label{s:Intro}
The theory of Anderson localization is concerned with spatial concentration and decay of eigenfunctions, as well as the corresponding dynamical quantities like wave-packets. 
The interest in these features stems from the quantum theory of disordered media, which provides a relation to transport properties of the modeled medium.
\par
A paradigmatic, and probably the most studied model in the mathematics literature on Anderson localization is the alloy-type potential.
The reason for its popularity is that it allows to model structural features of the stochastic field determining the potential explicitly. The features which have attracted most 
of attention are non-monotonicity, the covering-property, and long-range correlations.
\par
They have been addressed for instance in \cite{Klopp-95a}, \cite{Kirsch-96}, and \cite{KirschSS-98b}, respectively. 
A further advantage of the alloy-type model is, that it can be treated in the continuum as well as in the discrete setting.
In particular, the problem of non-monotonicity in the potential has been tackled in many papers, with intensified interest in recent years, 
\cite{Klopp-95a}, 
\cite{Veselic-01,Veselic-02a}, 
\cite{HislopK-02}, 
\cite{KloppN-09a}, 
\cite{KostrykinV-06}, 
\cite{Bourgain-09}, 
\cite{Veselic-10b},\cite{Veselic-10},
\cite{ElgartTV-10,ElgartTV-11}, 
\cite{Krueger-12},
\cite{CaoE-12}, \cite{ElgartSS}. 
\par
The main results of the paper at hand are the following:
\begin{itemize}
 \item A new Wegner estimate valid for discrete as well as continuum alloy-type models.
 \item Along the way we give an explicit construction of strictly positive linear combinations of translates of single-site potentials.
 \item Compatibility of non-monotonicity with long-range interactions in the multiscale analysis proof of interactions. This has been obtained by Kr\"uger in \cite{Krueger-12} for the discrete model.
 We show that this holds in the continuum case as well.
\end{itemize}
Our implementation of the multiscale analysis (MSA in the following) is a different one than Kr\"uger's. 
Rather that relying on \cite{Bourgain-09}, we use the strategy of \cite{KirschSS-98b}. This results in a much simpler version of the MSA. 
For the discrete alloy-type model we give a detailed proof, accessible also to non-specialists.
\par
For the continuum analogue we establish a result on the control of resonances, which allows to merge into the MSA presented in \cite{KirschSS-98b}.
\par
In a separate Section~\ref{sec:discussion} we discuss the methodical and physical implications of negative (sign-changing single-site potential) and long range (non-compactly supported single-site potential) 
correlations.
%
%% %%%%%%%%%%%%%%%%%%%%%%%%%%%%%%%%%%%%%%%%%%%%%%%%%%%%%%%%%%%%%%%%%%%%%%%%%%%%%%%%%%%
%--------------------------------------------------------------------------------
%
%          Model and Results
% %--------------------------------------------------------------------------------
%%
%%%%%%%%%%%%%%%%%%%%%%%%%%%%%%%%%%%%%%%%%%%%%%%%%%%%%%%%%%%%%%%%%%%%%%%%%%%%%%%%%
%%
%
\section{Model and results} \label{sec:model}
\subsection{Model and basic notation}
 We first introduce the continuous model.
The \emph{alloy-type model} is given by the family of Schr\"o{}dinger operators
\begin{equation*}
H_\omega \colonequals H_0 + V_\omega \ \text{on $L^2 (\RR^d)$}, \quad H_0 := -\Delta + V_0, \quad \omega \in \Omega ,
\end{equation*}
where $-\Delta$ is the negative Laplacian, $V_0$ a $\ZZ^d$-periodic potential, and $V_\omega$ denotes the multiplication by the $\ZZ^d$-ergodic random field
\begin{equation*}
V_\omega (x) \colonequals \sum_{k \in \ZZ^d} \omega_k U(x-k) .
\end{equation*}
The so-called \emph{coupling constants} $\omega_k$, $k \in \ZZ^d$, are assumed to be independent identically distributed (i.i.d.) random variables according to a probability measure $\mu$ with bounded support. Hence the probability space has a product structure $\Omega \colonequals \times_{k \in \ZZ^d} \supp \mu$, is equipped with the $\sigma$-Algebra $\mathcal{F}$ generated by the cylinder sets and the probability measure $\PP \colonequals \otimes_{k \in \ZZ^d} \mu$. 
The corresponding expectation is denoted by $\EE$, i.e.\ $\EE(\cdot) \colonequals \int_\Omega (\cdot) \PP (\drm \omega)$. For a set $\Gamma \subset \ZZ^d$, $\EE_\Gamma$ denotes the expectation with respect to $\omega_k$, $k \in \Gamma$. That is, $\EE_\Gamma (\cdot) \colonequals \int_{\Omega_\Gamma} \prod_{k \in \Gamma} \mu (\drm \omega_k)$ where $\Omega_{\Gamma} \colonequals \times_{k \in \Gamma} \RR$.
\par
The function $U:\RR^d \to \RR$ is called \emph{single-site potential}. Throughout this paper we assume that $V_0$ and $V_\omega$ are infinitesimally bounded with respect to $\Delta$ and that the corresponding constants can be chosen uniform in $\omega \in \Omega$. This is in particular satisfied if $U$ is a so-called \emph{generalized step-function}.
\begin{definition}[Generalized step-function] \label{def:step}
Let $L_{\rm c}^p (\RR^d) \ni w \geq \kappa \chi_{(-1/2,1/2)^d}$ with $\kappa > 0$ and $p = 2$ for $d \leq 3$ and $p>d/2$ for $d \geq 4$, where $L_{\rm c}^p (\RR)$ denotes the vector space of $L^p (\RR)$ functions with compact support. Let $u \in \ell^1 (\ZZ^d ; \RR)$. 
A function $U : \RR^d \to \RR$ of the form 
\[
U (x) = \sum_{k \in \ZZ^d} u (k) w(x-k) 
\]
is called \emph{generalized step-function} and the function $u:\ZZ^d \to \RR$ a \emph{convolution vector}. If $U$ is a generalized step-function we define $r = \sup\{ \lVert x \rVert_\infty \colon w(x) \not = 0 \}$.
\end{definition}
Recall that any real-valued function on $\RR^d$ that is uniformly locally $L^p$, with $p = 2$ for $d \leq 3$ and $p>d/2$ for $d \geq 4$, is infinitesimally bounded with respect to the self-adjoint Laplacian $\Delta$ on $W^{2,2} (\RR^d)$, see e.\,g. \cite[Theorem~XIII.96]{ReedS-78d}. 
This is indeed satisfied for $V_\omega$ if $U$ is a generalized step-function, since for any unit cube $C \subset \RR^d$ we have using $\supp u \subset [-r,r]^d$ and H\"older's inequality 
\begin{align}
\label{eq:loc-lp-norm}
\int_{C} \lvert V_\omega (x) \rvert^p \drm x &= \int_C \Bigl\lvert \sum_{k \in \ZZ^d} \omega_k \sum_{l \in \ZZ^d} u(l-k) w(x-l) \Bigr\rvert^p \drm x  \nn
\\ \nn
& \leq \omega_+^p \lVert u \rVert_{\ell^1 (\ZZ^d)}^p \int_C \Bigl( \sum_{l \in \ZZ^d} \chi_{[-r,r]^d} (x-l)  \lvert w(x-l) \rvert \Bigr)^p \drm x  \\ 
&\leq \omega_+^p \lVert u \rVert_{\ell^1 (\ZZ^d)} ^p (2r+1)^{d(p-1)}  \lVert w\rVert_{L^p (\RR)}^p ,
\end{align}
where $\omega_+ = \sup \{\lvert t \rvert : t \in \supp \mu\}$. Notice that the upper bound is uniform in $\omega \in \Omega$.
Hence, $V_0$ and $V_\omega$ are infinitesimally bounded with respect to $\Delta$ and the corresponding constants can be chosen uniform in $\omega \in \Omega$. Therefore, $H_\omega$ is self-adjoint (on the domain of $\Delta$) and bounded from below (uniform in $\omega \in \Omega$).
\par
Let us now introduce the discrete analogue of the alloy-type model. The \emph{discrete alloy-type model} is the family of Schr\"odinger operators 
\[
h_\omega\colonequals h_0+v_\omega \ \text{on $\ell^2(\ZZ^d)$}, \quad \omega \in \Omega .
\]
Here $h_0$ is the negative discrete Laplacian on $\ell^2 (\ZZ^d)$ given by 
\[
%\genfrac{}{}{0pt}{}{y \in \ZZ^d :}{\lVert x-y \rVert_1 = 1}
(h_0 \psi)(x) = \sum_{\lVert x-y \rVert_1 = 1} (\psi (x) - \psi (y)). 
\]
The random part $v_\omega$ is a multiplication operator by the function
\begin{equation*} 
v_\omega (x) \colonequals\sum_{k \in \ZZ^d} \omega_k \, u(x-k) ,
\end{equation*}
where $u \in \ell^1(\ZZ^d;\RR)$ is called \emph{single-site potential}. Notice that in the continuous setting, the (discrete) single-site potential $u$ plays the role of a convolution vector to generate the (continuous) single-site potential $U$ in form of a generalized step-function. With other words, the convolution vector of a generalized step-function serves as a single-site potential for our discrete model. 
\par
Next we introduce some assumptions on the function $u$ and the measure $\mu$. 
For $k \in \ZZ^d$ we denote by $\lVert k \rVert_1 \colonequals \sum_{r=1}^d \lvert k_r \rvert$ the $\ell^1$-norm of $k$.
\begin{assumption} \label{ass:exp}
There are constants $C,\alpha > 0$ such that for all $k \in \ZZ^d$ we have
\begin{equation*} %\label{eq:exponential}
\lvert u(k) \rvert \leq C \euler^{-\alpha \lVert k \rVert_1} .
\end{equation*}
\end{assumption}
Assumption~\ref{ass:exp} gives rise to constants $c_u \not = 0$ and $I_0 \in \NN_0^d$, both depending only on the function $u$, i.e.~$C$ and $\alpha$. The constants $c_u$ and $I_0$ are defined in Section~\ref{sec:transformation}, see in particular Eq.~\eqref{eq:cF}. We use the shorthand notation $N = \lVert I_0 \rVert_1$. If the mean value $\overline{u} = \sum_{k \in \ZZ^d} u(k)$ is positive one can choose $I_0 = 0$ and $c_u = \overline u$. 
If $\overline u = 0$ then $I_0$ and $c_u$ depend on the behavior of the generating function associated to $u$,  at the argument value $1\in \CC^d$. 
Moreover,  for any $l > 0$ we define
\begin{equation*}
  R_l \colonequals \max \left\{ 2l + \frac{2}{\alpha} \ln \frac{2\cdot 3^d\, C}
  {\lvert c_u \rvert (1-\mathrm{e}^{-\alpha/2})}, \frac{8 (d+\lVert I_0 \rVert_1)^2}{\alpha^2} \right\}. 
\end{equation*}
\begin{assumption}\label{ass:bv}
The measure $\mu$ has a density $\rho \in \BV (\RR)$.
\end{assumption}
Here $\BV(\RR)$ denotes the space of functions of finite total variation. A precise
definition of this function space is given in Section~\ref{sec:abstract_wegner}.
\begin{assumption} \label{ass:small_neg}
We say that Assumption~\ref{ass:small_neg} is satisfied for $\delta > 0$, if there exists a decomposition $u = u_+ - \delta u_-$ with $u_+ , u_- \in \ell^1 (\ZZ^d ; \RR^+_0)$, and $\lVert u_- \rVert_1 \leq 1$. 
For the measure $\mu$ we assume $\supp \mu = [0,\omega_+]$ for some $\omega_+ > 0$.
\end{assumption}
The estimates we want to prove concern finite box restrictions  of the operator $H_\omega$ or $h_\omega$, $\omega\in  \Omega$. For $l>0$ and $j \in \ZZ^d$ we denote by 
\[
\Lambda_{l} (j)\colonequals (-l , l)^d + j \subset \RR^d
\]
the open cube of side length $2l$ centered at $j$. We will use the notation $\Lambda_{l} \colonequals \Lambda_{l} (0)$. By $H_\omega^\Lambda$ we denote the restriction of the operator $H_\omega$ to a bounded open set $\Lambda \subset \RR^d$ with Dirichlet boundary conditions on $\partial \Lambda$. In the special case when $\Lambda$ is a cube, $H_\omega^\Lambda$ will denote the restriction of $H_\omega$ to $\Lambda$ either with Dirichlet or with periodic boundary conditions. Let $P_B (H_\omega^\Lambda)$ denote the spectral projection for the operator $H_\omega^\Lambda$ associated with a Borel set $B \subset \RR$. If $\Lambda = \Lambda_{l}$ we will write $H_\omega^l$ and $P_B (H_\omega^l)$ instead of $H_\omega^{\Lambda_{l}}$ and $P_B (H_\omega^{\Lambda_{l}})$.
%\par
Analogously for the discrete model, for $l > 0$ and $j \in \ZZ^d$ let 
\[
\cd{l} (j) \colonequals \bigl([-l,l]^d + j\bigr) \cap \ZZ^d 
\]
and $\cd{l} \colonequals \cd{} (0)$. For $\cd{} \subset \ZZ^d$ finite we denote the canonical inclusion $\ell^2(\cd{})\to \ell^2(\ZZ^d)$ by $\iota_{\cd{}}$ and the adjoint restriction $\ell^2(\ZZ^d)\to \ell^2(\cd{})$ by $\pi_{\cd{}}$. The restriction of $h_\omega$ to $\cd{}$ is defined by $h_{\omega}^{\cd{}} \colonequals \pi_{\cd{}} h_0 \iota_{\cd{}} + \pi_{\cd{}} v_\omega \iota_{\cd{}} \colon \ell^2(\cd{})\to \ell^2(\cd{})$. 
Let $P_B (h_\omega^{\cd{}})$ denote the spectral projection for the operator $h_\omega^{\cd{}}$ associated with a Borel set $B \subset \RR$ and if $l > 0$ and $\cd{} = \cd{l}$ we will write $h_\omega^l$ and $P_B (h_\omega^l)$ instead of $h_\omega^{\cd{l}}$ and $P_B (h_\omega^{\cd{l}})$.
\subsection{Results on Wegner estimate}
Now we are in the position to state our bounds on the expected number of eigenvalues of finite box Hamiltonians $H_{\omega}^l$ and $h_{\omega}^l$ in a bounded energy interval $[E-\epsilon, E+\epsilon] \subset \RR$. They are called Wegner estimates \cite{Wegner-81} and are inequalities of the type
\begin{equation} \label{eq:generalWegner}
\forall\, l>0, E \in \RR,\epsilon>0 \colon \quad
\EE \bigl \{\Tr \bigl( P_{[E-\epsilon, E+\epsilon]} (h_\omega^l) \bigr) \bigr\} \le C_{\rm W} \ (2 \epsilon)^a\, (2l+1)^{b\, d}
\end{equation}
with some (Wegner-)constant $C_{\rm W}$, some $a\leq1$ and some $b\geq 1$. The exponent $a$ determines the quality of the estimate with respect to the \emph{length of the energy interval} and $b$ the quality with respect to the \emph{volume of the cube} $\cd{l}$. The best possible estimate is obtained in the case $a=1$ and $b=1$. 
\par
The precise formulation of the Wegner estimate relies on the definition of the quantities $r$, $c_u$, $I_0$, $N$ and $R_l$ 
introduced above and defined precisely in Section~\ref{sec:transformation}. 
\begin{theorem}[Wegner estimate, continuous model] \label{theorem:wegner_c} 
Assume that $U$ is a generalized step-function and that Assumptions~\ref{ass:exp} and \ref{ass:bv} are satisfied.
Then there exists $C_{\rm W} = C_{\rm W}(U)>0$ and $N = N(u)$, such that for any $l>0$ and any bounded interval $I\colonequals[E_1,E_2] \subset \RR$ we have with $\Gamma = \cd{R_{l+r}}$
\[
\EE_\Gamma \bigl \{\Tr \bigl(P_I (H_\omega^l) \bigr)\bigr\}
\le
\euler^{E_2} C_{\rm W} \lVert \rho \rVert_{\rm Var} \lvert I \rvert (2l+1)^{2d + N} .
\]
\end{theorem}
\begin{theorem}[Wegner estimate, discrete model] \label{theorem:wegner_d}
Let Assumptions~\ref{ass:exp} and \ref{ass:bv} be satisfied. Then there exists $C_{\rm W} = C_{\rm W}(u)>0$ and $N = N(u)$ such that for any $l>0$ and any bounded interval $I \subset \RR$ we have with $\Gamma = \cd{R_l}$
\[
 \EE_\Gamma \bigl \{\Tr \bigl(P_I (h_\omega^l) \bigr)\bigr\}\le
C_{\rm W} \lVert \rho \rVert_{\rm Var} \lvert I \rvert (2l+1)^{2d + N} .
\]
\end{theorem}
Of course the same bound follows for the full expectation $\EE \{\Tr (P_I (h_\omega^l) ) \}$ and $\EE \{\Tr (P_I (H_\omega^l) ) \}$, respectively. However, in our application it is crucial to be able to work with the partial average $\EE_\Gamma$.
The main point of Theorem~\ref{theorem:wegner_c} and \ref{theorem:wegner_d} is that no assumption on $u$ (apart from exponential decay) is required. 
In particular, the sign of $u$ can change arbitrarily.
Also, note that the result holds on the whole energy axis.
\begin{remark}[Continuous model]
Theorem~\ref{theorem:wegner_c} has been already made available in the preprint \cite{PeyerimhoffTV-11}.
It generalizes or is complementary to earlier results on Wegner estimates.
Let us compare the result of Theorem~\ref{theorem:wegner_c} to earlier ones of Wegner estimates for alloy-type models with sign-changing single-site potential. 
\par
The papers \cite{Klopp-95a,HislopK-02} concern alloy-type Schr\"odinger operators on $L^2(\RR^d)$.
The main result is a Wegner estimate for energies in a neighborhood of the infimum of the spectrum. It applies to arbitrary non-vanishing single-site potentials $u \in C_{\rm c}(\RR^d)$ and coupling constants with a piecewise absolutely continuous density. The upper bound is linear in the volume of the box and H\"older-continuous in the energy variable.
\par
The papers  \cite{Veselic-02a,KostrykinV-06,Veselic-10} 
establish Wegner estimates for both alloy-type Schr\"o\-dinger operators on $L^2(\RR^d)$
and discrete alloy-type models on $\ell^2(\ZZ^d)$.
We will discuss now the results of \cite{Veselic-02a,KostrykinV-06,Veselic-10} referring to operators on $L^2 (\RR^d)$. These papers give Wegner estimates that are linear in the volume of the box and Lipschitz continuous in the energy variable.
The bounds are valid for all compact intervals along the energy axis.
They apply to single-site potentials $U \in L_c^\infty(\RR^d)$ of a generalized step function form 
with a convolution vector satisfying 
\begin{equation}
\label{eq:symbol}
s\colon \theta\mapsto s(\theta) \colonequals\sum_{k \in \ZZ^d} u(k) \mathrm \euler^{-\mathrm i k \cdot \theta}
\text{ does not vanish on $[0, 2\pi  )^d $. }
\end{equation}
\end{remark}
\begin{remark}[Discrete model]
Theorem~\ref{theorem:wegner_d} has been made publicly available in the preprint \cite{PeyerimhoffTV-11} and
generalizes the results established in \cite{Veselic-10b}. 
There the same result as in Theorem~\ref{theorem:wegner_d} has been established, under the additional assumption that at least one of the following conditions holds:
\begin{enumerate}[(i)]
 \item $\bar u \colonequals \sum_{k\in\ZZ^d} u(k) \neq 0$, or
 \item $u$ is finitely supported, or
 \item the space dimension satisfies $d=1$.
\end{enumerate}
If condition (i) is satisfied, the Wegner bound of \cite{Veselic-10b} holds for all $u \in \ell^1(\ZZ^d)$ not necessarily of exponential decay.
The volume dependence of the upper bound in the Wegner estimate in \cite{Veselic-10b} is slightly better than ours here.
A particularly important case in \cite{Veselic-10b} is the one when both conditions (i) and (ii) hold. 
In this situation the exponent of the length scale can be chosen to be equal to the space dimension $d$. 
This corresponds to the volume exponent $b=1$ in Ineq.~\eqref{eq:generalWegner}, 
and yields the Lipschitz continuity of the integrated density of states. 
This is the distribution function $N\colon \RR \to \RR$ obtained as the limit 
\[
\lim_{l\to \infty} \frac{1}{(2l+1)^d} \EE \left \{\Tr \bigl (P_{(-\infty,E]}(h_{\omega}^l)\bigr)\right\}
\equalscolon N(E)
\]
at all continuity points of $N$. Consequently its derivative, the density of states, exists for almost all $E \in \RR$.
While in certain situations where $\bar u =0$, a Wegner estimate was already established in  \cite{Veselic-10b}, 
the improved proof presented here allows more explicit control of the exponent $b$. The proof in \cite{Veselic-10b} uses an induction argument over the space dimension, 
which obscures certain parameter dependencies.
\par
In \cite{ElgartSS} a Wegner estimate for compactly supported single-site potentials with H\"older continuous distributions was proven.
\end{remark}
\subsection{Results on localization}
The Wegner estimates from the previous subsection may be used as an ingredient for the multiscale analysis. Localization then follows once an appropriate initial scale estimate is satisfied. To keep things short, we restrict ourselves to the discrete model. Similar results may be obtained for the continuous model as well, see the discussion in Section~\ref{sec:loc_cont}. Our main results on localization for the discrete model are formulated in the following three theorems. All of them are proven in Section~\ref{sec:loc_discrete}.
\begin{theorem}[Output of multiscale analysis] \label{thm:loc_under_ini}
 Let Assumptions~\ref{ass:exp} and \ref{ass:bv} be satisfied, $I \subset \RR$ and assume that the initial scale estimate, as formulated in Definition~\ref{ass:ini}, holds in $I$. 
 \par
 Then, for almost all $\omega\in\Omega$, $\sigma_{c} (h_\omega) \cap I = \emptyset$ and the eigenfunctions corresponding to the eigenvalues of $h_\omega$ in $I$ decay exponentially.
\end{theorem}
Theorem~\ref{thm:loc_under_ini} is an adaptation of the multiscale analysis \`a la \cite{KirschSS-98b} 
for the discrete alloy-type model with sign-changing long-range single-site potentials, 
once the Wegner estimate from Theorem~\ref{theorem:wegner_d} is available. 
\par
Once the initial scale estimate is verified, Theorem~\ref{thm:loc_under_ini} gives localization. The proof of Theorem~\ref{thm:loc_under_ini} is based on multiscale analysis in the manner of \cite{DreifusK-89,KirschSS-98b} using our Wegner estimate as an ingredient, and finally applying Theorem~2.3 of \cite{DreifusK-89}. 
More precisely, we need a slight generalization of this theorem since they consider the i.i.d.\ Anderson model only. 
We show in fact that Theorem~2.3 of \cite{DreifusK-89} stays valid for general families of self-adjoint operators, cf.\ Theorem~\ref{thm:vDK-2.3}. 
\par
If the single-site potential $u$ has fixed sign, the validity of the initial scale estimate is well known, see e.g.\ \cite{KirschSS-98a,KirschSS-98b}. If the single-site potential changes sign, far less is known. 
We prove the initial scale estimate in the case of large disorder for all energies if the single-site potential decays exponentially. 
The proof is in the manner of \cite{Kirsch-08,Veselic-10b} and is based on the so-called uniform control of resonances, which we provide for our model in Section~\ref{sec:resonances}. 
At the infimum of the spectrum  we verify the initial scale estimate if the single-site potential has a small negative part, including the case of unbounded support of the single-site potential. 
This is a generalization of \cite{Veselic-01,Veselic-02a}, where similar results have been shown for compactly supported single-site potentials. 
The paper \cite{CaoE-12} proves the an initial length scale estimate at weak disorder for exponentially decaying sign-changing single-site potentials in the 
case $d = 3$. 
However, this result can be applied as an ingredient for the multiscale analysis only for compactly supported single-site potentials, 
since they prove a non-uniform version only. In this case a multiscale analysis requires independence at distance.
\begin{theorem}[Localization, large disorder] \label{thm:loc:large}
 Let Assumptions~\ref{ass:exp} and \ref{ass:bv} be satisfied and $\lVert \rho \rVert_{\rm BV}$ be sufficiently small. Then, for almost all $\omega\in\Omega$, $\sigma_{c} (h_\omega) = \emptyset$ and the eigenfunctions corresponding to the eigenvalues of $h_\omega$ decay exponentially.
\end{theorem}
\begin{theorem}[Localization, small negative part of $u$] \label{thm:loc:weak}
Let Assumptions~\ref{ass:exp} and \ref{ass:bv} be satisfied and $\overline{u} > 0$. Then there exists $\delta > 0$ and $\epsilon > 0$, such that if Assumption~\ref{ass:small_neg} is satisfied for $\delta$, then, for almost all $\omega \in \Omega$, $\sigma_{\rm c} (h_\omega) \cap [- \epsilon , \epsilon] = \emptyset$ and the eigenfunctions corresponding to the eigenvalues of $h_\omega$ decay exponentially.
\end{theorem}
\begin{remark}[Localization, continuous model]
Based on the Wegner estimate from Theorem~\ref{theorem:wegner_c} and the uniform control of resonances from Proposition~\ref{prop:unif_wegner_type_cont} for the continuous model, similar results to those of Theorem~\ref{thm:loc_under_ini}, \ref{thm:loc:large} and \ref{thm:loc:weak} follow for the continuous alloy-type model on $L^2 (\RR^d)$ with a generalized step function as a single-site potential. 
This is discussed in detail in Section~\ref{sec:loc_cont}.
\end{remark}
\begin{remark}
Kr\"uger \cite{Krueger-12} has obtained results on localization for a class of discrete alloy-type models which includes the 
ones considered here.
The results rely  on the multiscale analysis and the use of Cartan's 
lemma in the spirit as is has been used earlier, e.g.\ in \cite{Bourgain-09}.
\par
There has also been progress for discrete alloy-type models for localization proofs via the fractional moment method. 
In \cite{TautenhahnV-10} boundedness of fractional moments of the Greens function was established under condition \eqref{eq:symbol}.
In \cite{ElgartTV-10,ElgartTV-11} finite volume criteria for a localization proof via the fractional moment method were established.
The best results so far in this setting have been obtained in the paper \cite{ElgartSS}.
There fractional moment bounds and localization for a class of matrix valued Anderson models, as well as discrete alloy-type models has been derived.
\end{remark}
%
%% %%%%%%%%%%%%%%%%%%%%%%%%%%%%%%%%%%%%%%%%%%%%%%%%%%%%%%%%%%%%%%%%%%%%%%%%%%%%%%%%%%%
%--------------------------------------------------------------------------------
%
%          Discussion
% %--------------------------------------------------------------------------------
%%
%%%%%%%%%%%%%%%%%%%%%%%%%%%%%%%%%%%%%%%%%%%%%%%%%%%%%%%%%%%%%%%%%%%%%%%%%%%%%%%%%
%%
%

\section{Non-monotonicity and long range interactions}
\label{sec:discussion}
The results of the paper address several challenges 
present in the theory of Anderson localization. Here we will discuss 
these aspects first separately, and then how they interact. We will 
make clear how this is manifested in the methodical implementation of the proof on the one hand, and the underlying physical phenomena on the other. 
\par
While in the theory of Anderson localization one is ultimately interested 
in spatial concentration and decay of eigenfunctions, as well as 
corresponding dynamical quantities like wave packets, a crucial 
intermediate stage of the analysis concerns the regularity of the 
distribution of spectral data.
\par
More precisely, one needs to understand how the regularity properties 
of the stochastic process defining the random potential, or more precisely 
of its distribution measure, translate into regularity of the distribution 
of various spectral quantities. Particular aspects of the random potential 
of alloy-type having effect on regularity of spectral data are the 
following:
\subsection*{Sign change of the single-site potential and negative correlations}
The first Wegner estimates and localization proofs for alloy-type potentials 
concerned non-negative single-site potentials. This implies that the quadratic form associated with 
the random Hamiltonian and thus the eigenvalues are monotone functions of the coupling constants. This facilitates rather explicit averaging estimates. Later on Wegner estimates were 
developed for sign-changing single-site potentials as well, see the discussion in 
Section~\ref{s:Intro}. 
The problem of non-monotonicity is not just a technical effect of the methods of proof. 
For instance, a random perturbation of vanishing mean will not induce effective averaging of spectral data on the level of first order perturbation theory. 
This is reflected in Theorem~\ref{theorem:wegner_d} concerning a Wegner estimate. The constant $C_u$ there captures the vanishing and non-vanishing of moments associated to the single-site perturbation. To overcome the problems 
posed by non-monotonicity we employ averaging over local environments:
While in the monotone situation typically averaging over the random variables associated to matrix coefficients of a Hamiltonian on a finite volume subsystem is sufficient to regularize 
spectral data, we average over the surrounding randomness as well. 
\subsection*{Regularity of the single-site marginal distribution}
It goes without saying that this feature is crucial for the smoothness of the distribution of spectral data. 
\par
From our results concerning sign-changing single-site potentials we have to require more regularity than for the analogous statements with semidefinite perturbations. While this is a 
restriction it appears naturally in this context, see e.g.\ \cite{Klopp-95a, Veselic-01, Veselic-02a, HislopK-02}. It is quite remarkable that 
\cite{Krueger-12} implements a multiscale 
analysis which does not need weak differentiability of the single-site distribution. However, Kr\"uger obtains a weaker Wegner estimate than ours.
The paper \cite{ElgartSS} can even treat H\"older continuous single-site distributions in the framework of the fractional moment method, for the strong disorder regime.
\subsection*{Long range correlation}
A single-site potential $u$ of non-compact support induces correlations at infinite distances. While this does not affect the proof of the Wegner estimate itself, it does its use in the 
multiscale analysis. A sufficiently fast decay of the single-site potential ensures that resonant energies in different regions are sufficiently decorrelated, cf.\ \cite{KirschSS-98b} 
and Section~\ref{sec:resonances} here.
\subsection*{Simultaneous long-range and negative correlations}
In the literature on alloy-type models sign-changing single-site potentials (resulting in negative correlations of the stochastic field $V_\omega(x), x\in \mathbb{R}^d$) 
and non-compactly supported 
single-site potentials (resulting in long range correlations of the stochastic field) are treated separately, see however \cite{Krueger-12}.
\par
In the models we consider, both on the lattice and the continuum, both difficulties are present, which posed an additional challenge for the understanding of the physical phenomena 
leading to averaging of spectral data and ultimately to localization. 
On first sight one might ask whether the two difficulties are intertwined at all or can they simply be treated separately.
\par
It turns out that this is not the case, as our proof shows, reflecting an inherent mechanism of non-monotone averaging. To obtain a Wegner estimate for a Hamiltonian on a finite box we need 
to average not only over the random variables contained in the box, but over a surrounding belt as well, growing with the size of the box.
\par
This shows that although negative correlations are a-priori a local phenomenon, they cannot be controlled without taking care of long range correlations of the potential as well. 
Furthermore, non-monotonicity of our model results in 
a weaker Wegner estimate (the volume term appears with a controllable, but possibly high power). This is the reason that we can treat only exponentially decaying sign-changing single-site 
potentials. An extension of our results to a restricted class of polynomially decaying potentials is possible, but is very technical. 
\par
In other work \cite{TautenhahnV-10b, TautenhahnV-13, TautenhahnV-13b} two of us
have derived complementary results concerning alloy-type models on the lattice, which shed new light on the result of this paper.
The contributions \cite{TautenhahnV-10b, TautenhahnV-13} concern the question to what extend discrete alloy-type potentials are a good way to model correlated fields on the lattice. There exist abstract conditions 
\cite{DreifusK-91, AizenmanFSH-01} on stochastic fields on $\mathbb{Z}^d$ which ensure that a proof of localization can be carried out following the multiscale analysis or the fractional moment method. 
They apply to a class of Gaussian fields (but not all).
\par
However, for discrete alloy-type potentials the abstract conditions are not satisfied, as long as the coupling constants 
and the support of the single-site potential are bounded. Moreover, discrete alloy-type 
potentials are used in other areas of mathematics under the name of multidimensional moving average processes. 
So, they can be considered as a class of its own for modeling correlated potentials 
of the Anderson model.
\par
In this paper we establish for a class of random Hamiltonians, among others, exponential localization in a specified energy region, i.e.\ almost sure exponential decay of eigenfunctions. 
In the papers \cite{Krueger-12,ElgartSS} for a closely related class of models on the lattice dynamical localization was established. In the continuum setting there is a general paradigm \cite{GerminetK-01} 
that various forms of localization, including exponential and dynamical coincide under natural assumptions on the model. For this reason one denotes the corresponding energy interval as the 
region of complete localization.
\par
In the physics community localization is interpreted in terms of decay of correlations (of Green's and eigenfunctions), in terms of the inverse participation ratio, and eigenvalue statistics.
 Thus it is desirable that the region of complete localization can be characterized in terms of these notions as well. 
\par
There have been efforts and  (modest) progress in this direction. Specifically, starting with the paper \cite{Molchanov-81} 
there have been mathematically rigorous results on Poisson statistics 
of eigenvalues. While \cite{Molchanov-81} concerns a continuum model in one dimension, \cite{Minami-96} established Poisson-statistics in the localized regime for the Anderson model on $\mathbb{Z}^d$ 
in arbitrary dimension. Subsequently there was further progress in this direction, cf.\ \cite{GrafV-07, BellissardHS-07, CombesGK-09, GerminetK-13, GerminetK-11b}. In \cite{TautenhahnV-13b} we showed that the result of 
Minami can be extended to a certain class of discrete alloy-type, and Poisson statistics of eigenvalues follows along \cite{GerminetK-13}.
\par
While the class of potentials treated in \cite{TautenhahnV-13b} is quite restricted, it is the first rigorous result, a part from one-dimensional ones, where the random variables couple to a perturbation which is 
not of rank one.
%
%% %%%%%%%%%%%%%%%%%%%%%%%%%%%%%%%%%%%%%%%%%%%%%%%%%%%%%%%%%%%%%%%%%%%%%%%%%%%%%%%%%%%
%--------------------------------------------------------------------------------
%
%          Positive combinations
% %--------------------------------------------------------------------------------
%%
%%%%%%%%%%%%%%%%%%%%%%%%%%%%%%%%%%%%%%%%%%%%%%%%%%%%%%%%%%%%%%%%%%%%%%%%%%%%%%%%%
%%
%
\section{Positive combinations of translated single-site potentials}
\label{sec:transformation}
In this section we consider (possibly infinite) linear combinations of translates of the (discrete) single-site potential $u$. In this section we assume that Assumption \ref{ass:exp} holds, i.e.\ there are constants $C,\alpha > 0$ such that
\begin{equation} \label{eq:exponential}
\lvert u(k) \rvert \leq C \euler^{-\alpha \lVert k \rVert_1} ,
\end{equation} 
and that $u$ is distinct from the zero function. Under these hypotheses we identify a sequence of coefficients such that the resulting linear combination is uniformly positive on the whole space $\ZZ^d$ (cf.\ Proposition~\ref{prop1}) or some finite subset of $\ZZ^d$ (cf.\ Proposition~\ref{prop2}).
\par
First we introduce the following multi-index notation: If $I
=(i_1,\dots,i_d) \in \ZZ^d$ and $z \in \CC^d$, we define
\[
z^I = z_1^{i_1} \cdot z_2^{i_2} \cdot \ldots \cdot z_d^{i_d},
\]
and if $I \in \NN_0^d$, we define
\begin{align*}
  D_z^I = \frac{\partial^{i_1}}{\partial z_1^{i_1}} \cdot \frac{\partial^{i_2}}{\partial z_2^{i_2}} \cdot \ldots \cdot \frac{\partial^{i_d}}{{\partial z_d}^{i_d}}, \quad
  I! = i_1! \cdot i_2! \cdot \ldots \cdot i_d! . 
\end{align*}
We also introduce comparison symbols for multi-indices: If $I, J \in \NN_0^d$, we write $J \le I$ if we have $j_r \le i_r$ for all
$r=1,2,\ldots,d$, and we write $J < I$ if $J \le I$ and $\lVert J \rVert_1 <
\lVert I \rVert_1$. For $J \le I$, we use the short hand notation
\[ 
\binom{I}{J} = \binom{i_1}{j_1} \cdot \binom{i_2}{j_2} \cdot \ldots \cdot \binom{i_d}{j_d}. 
\]
Finally, ${\mathbf 0}, {\mathbf 1}$ denote the vectors $(0,\dots,0)$
and $(1,\dots,1) \in \CC^d$, respectively.
\par
We also recall the following facts from multidimensional complex
analysis. Let $D \subset \CC^d$ be open. We call a complex valued function $f : D \to \CC$ holomorphic, if every point $w \in D$ has an open neighborhood $U$, $w \in U \subset D$, such that $f$ has a power series expansion around $w$, which converges to $f(z)$ for all $z \in U$.
Osgood's lemma tells us that, if $f : D \to \CC$ is continuous and holomorphic in each variable separately (in the sense of one-dimensional complex analysis), then $f$ is holomorphic, see \cite{GunningR-09}.
Let $f_n : D \to \CC$ be a sequence of holomorphic functions. We say that $\sum_n f_n$ converges normally in $D$, if for every $w \in D$ there is an open neighborhood $U$, $w \in U \subset D$, such that $\sum_n \lVert f_n \rVert_{U,\infty} < \infty$. Normally convergent sequences of holomorphic functions can be rearranged arbitrarily, the limit is again holomorphic, and differentiation can be carried out termwise, which follows from Weierstrass' theorem, see \cite[p.\ 226]{Remmert-84} for the one-dimensional case and \cite[p.\ 7]{Narasimhan-95} for the higher dimensional case.
\par
For $\delta \in (0, 1-\euler^{-\alpha})$ we consider the to $u$ associated generating function $F : D_\delta \subset \CC^d \to \CC$, 
\begin{equation*} %\label{eq:Fz}
D_\delta = \{ z \in \CC^d : \lvert z_1 - 1 \rvert < \delta , \ldots , \lvert z_d - 1 \rvert < \delta \}, \quad F(z) = \sum_{k \in \ZZ^d} u(-k) z^k .
\end{equation*}
Notice that the sum $\sum_{k \in \ZZ^d} u(-k) z^k$ is normally convergent in $D_\delta$ by our choice of $\delta$ and the exponential decay condition \eqref{eq:exponential}. By Weierstrass' theorem, $F$ is a holomorphic function. Since $F$ is holomorphic and not identically zero, we have $(D_z^I F) (\mathbf{1}) \not = 0$ for at least one $I \in \NN_0^d$. Therefore, there exists a multi-index $I_0 \in \NN_0^d$ (not necessarily unique), such that we have
\begin{equation} \label{eq:cF}
  (D_z^I F)({\mathbf 1}) = 
		\begin{cases} 
			c_u \neq 0, & \text{if $I = I_0$,} \\
			0,          & \text{if $I < I_0$.} 
		\end{cases}  
\end{equation}
Such a $I_0$ can be found by diagonal inspection: Let $n \ge 0$ be the largest integer such that $D_z^IF({\mathbf 1}) = 0$ for all $\lVert I \rVert_1 < n$.  Then choose a multi-index $I_0 \in \NN_0^d$, $\lVert I_0 \rVert_1 = n$ with $(D_z^{I_0} F)({\mathbf 1}) \neq 0$.
\begin{proposition} \label{prop1}
Let $u$, $c_u$ and $I_0$ be as in \eqref{eq:exponential} and \eqref{eq:cF}. Let further $I \in \NN_0^d$ with $I \le I_0$, and define $a: \ZZ^d \to \ZZ$ by 
\[
a(k) = k^I.
\]
Then we have for all $x \in \ZZ^d$
\begin{equation} \label{eq:akuxk} 
    \sum_{k \in \ZZ^d} a(k) u(x-k) = \begin{cases} 0, & \text{if $I < I_0$,} \\
    c_u, & \text{if $I = I_0$.} \end{cases}
\end{equation}
\end{proposition}
\begin{proof}
We introduce, again, a bit of notation. For $s \in \CC^d$ and $k \in \ZZ^d$ let
\[ 
  \mathrm{e}^s  =  (\mathrm{e}^{s_1},\dots,\mathrm{e}^{s_d}) \quad \text{and} \quad \langle k,s \rangle  =  \sum_{r=1}^d k_r s_r.
\] 
Let $I \le I_0$. Then the chain rule yields (for all $s \in C_\delta \colonequals \{s \in \CC^d : {\rm e}^s \in D_\delta\}$)
\begin{align*}
D_s^I(F(\mathrm{e}^s)) &= \sum_{J \le I} c_J\, (D_z^J F)(\mathrm{e}^s)\, \mathrm{e}^{\langle J,s \rangle} \\
&=(D_z^I F)(\mathrm{e}^s)\, \mathrm{e}^{\langle I,s \rangle}+  \sum_{J <I} c_J\, (D_z^J F)(\mathrm{e}^s)\, \mathrm{e}^{\langle J,s \rangle},
\end{align*}
with suitable integers $c_J \ge 1$ and, in particular, $c_I = 1$. This and Eq.~\eqref{eq:cF} imply that
\begin{equation} \label{eq:diffFes}
  D_s^I(F(\mathrm{e}^s))\big\vert_{s={\mathbf 0}} = \begin{cases} 0, & \text{if
      $I < I_0$,}\\ c_{I_0}\, (D^{I_0}_zF)({\mathbf 1}) = c_u, & \text{if
      $I = I_0$.} \end{cases}
\end{equation}
Next, we use the identity $a(k) = k^I = D_s^I \mathrm{e}^{\langle k,s   \rangle}\big\vert_{s={\mathbf 0}}$. Note that the series $\sum_{k\in \ZZ^d} u(x-k)\mathrm{e}^{\langle k,s \rangle}$ converges normally on the domain
\[
E_\alpha = \{ s \in \CC^d \mid - \alpha < {\rm Re}(s_j) < \alpha \ \text{for all $j=1,2,\dots,d$} \} .
\]
Therefore, we can rearrange arbitrarily, differentiate componentwise, and obtain for all $s \in C_\delta \cap E_\alpha$ by substitution $\nu = k-x$ and the product rule
\begin{align*}
\sum_{k \in \ZZ^d} u(x-k) D_s^I \mathrm{e}^{\langle k,s \rangle} &= D_s^I \sum_{k \in \ZZ^d} u(x-k) \mathrm{e}^{\langle k,s \rangle} \\
&= D_s^I \Bigl( \mathrm{e}^{\langle x,s \rangle}  \sum_{\nu \in \ZZ^d} u(-\nu) \mathrm{e}^{\langle \nu,s \rangle} \Bigr) = D_s^I \Bigl( F(\mathrm{e}^s) \mathrm{e}^{\langle x,s \rangle} \Bigr) \\
&= \sum_{J \le I} \binom{I}{J} \bigl( D_s^J F(\mathrm{e}^s) \bigr) D_s^{I-J} \mathrm{e}^{\langle x,s \rangle} . 
\end{align*} 
Finally, evaluating at $s = \mathbf 0$ and using \eqref{eq:diffFes} yields
\begin{align*}
\sum_{k \in \ZZ^d} a(k) u(x-k) &= \sum_{J \le I} \binom{I}{J} \bigl( D_s^J F(\mathrm{e}^s)\bigr) \big\vert_{s={\mathbf 0}} \bigl( D_s^{I-J}(\mathrm{e}^{\langle x,s \rangle} \bigr)\big\vert_{s={\mathbf 0}} \\
&= \begin{cases} 0, & \text{if $I < I_0$}, \\
        c_u, & \text{if $I = I_0$.} \end{cases} \qedhere
\end{align*}
\end{proof}
In Proposition~\ref{prop1} we identified a sequence of coefficients such that the associated linear combination of translated single-site potentials is positive on the whole of $\ZZ^d$. However, the sequence cannot be used for Theorem~\ref{theorem:abstract1} and \ref{theorem:abstract2} directly. This problem can be resolved if we take into consideration that the positivity assumption in Theorem~\ref{theorem:abstract1} and \ref{theorem:abstract2} concerns lattice sites in $\Lambda_l$ respectively $\cd{l}$ only.
\par
Recall that the constants $d, \alpha, C$ and $c_u$ are all determined by the choice of the exponentially decreasing function $u: \ZZ^d \to \RR$. Now we choose $I=I_0$ in Proposition~\ref{prop1}.
The next proposition tells us, for all integer vectors $x$ in the box $\cd{l}$, how far we have to exhaust $\ZZ^d$ in the sum \eqref{eq:akuxk}, in order to guarantee that the result is $\ge
c_u / 2$ (assuming for a moment that $c_u > 0$). The exhaustion is described by the integer indices in another box $\cd{R}$, and the proposition describes the relation between the sizes $l$ and $R$. For large enough $l$, this relation is linear.
\begin{proposition} \label{prop2}
Let $u$, $c_u$ and $I_0$ be as in \eqref{eq:exponential} and \eqref{eq:cF}. Let further $l > 0$ and define
\begin{equation} \label{eq:RLrel} 
  R_l \colonequals \max \left\{ 2l + \frac{2}{\alpha} \ln \frac{2\, 3^d\, C}
  {\lvert c_u \rvert (1-\mathrm{e}^{-\alpha/2})}, \frac{8 (d+\lVert I_0 \rVert_1)^2}{\alpha^2} \right\}. 
\end{equation}
Then we have for all $x \in \cd{l}$
\[ 
  \frac{2}{c_u} \sum_{k \in \cd{R_l}}  k^{I_0}\, u(x-k) \ge 1.
\] 
\end{proposition}
\begin{proof}
We know from Proposition~\ref{prop1} that 
\[
\frac{1}{c_u} \sum_{k \in \ZZ^d} k^{I_0}\, u(x-k) = 1, 
\]
for all $x \in \ZZ^d$. Thus we need to prove, for $x \in \cd{l} = \ZZ^d \cap [-l,l]^d$, that
\begin{equation} \label{eq:desired}
\Big  | \sum_{k \in \ZZ^d \setminus \cd{R_l} } k^{I_0}\, u(x-k)\Big | \le \frac{\lvert c_u \rvert}{2}.
\end{equation}
Using the triangle inequality $\Vert x-k \Vert_\infty + \Vert x \Vert_\infty \ge \Vert k \Vert_\infty$, $\Vert k \Vert_\infty \le \Vert k \Vert_1$, and that $u$ is exponentially  decreasing, we obtain
\begin{align*}
\Bigl\lvert \sum_{k \in \ZZ^d \setminus \cd{R_l} } k^{I_0}\, u(x-k) \Bigr\rvert 
&\le C \mathrm{e}^{\alpha \Vert x \Vert_\infty} \sum_{k \in \ZZ^d \setminus \cd{R_l}} 
    \Vert k \Vert_\infty^{\lVert I_0 \rVert_1} \, \mathrm{e}^{-\alpha \Vert k \Vert_\infty} \\
&\le C \mathrm{e}^{\alpha l} \sum_{r=\lceil R_l \rceil}^\infty (2r+1)^d \, r^{\lVert I_0 \rVert_1} \, \mathrm{e}^{-\alpha r} \\
&\le C 3^d \mathrm{e}^{\alpha l} \sum_{r= \lceil R_l \rceil}^\infty r^{d+\lVert I_0 \rVert_1} \mathrm{e}^{-\alpha r}.
\end{align*}
Here $\lceil x \rceil = \min\{ k \in \ZZ \colon k \geq x \}$. Using Lemma~\ref{lem:nexp} below and $r \ge R_l
   \ge [8(d+\lVert I_0 \rVert_1)^2]/\alpha^2$, we conclude that $r^{d+\lVert I_0 \rVert_1}
   \le \mathrm{e}^{\alpha r/2}$, which implies that
   \[
  \Big| \sum_{k \in \ZZ^d \setminus \cd{R_l} } k^{I_0}\, u(x-k)\Big |
   \le
  C 3^d \mathrm{e}^{\alpha l} \sum_{r=\lceil R_l\rceil}^\infty \mathrm{e}^{-\alpha r/2} = C 3^d
  \mathrm{e}^{\alpha l} \frac{\mathrm{e}^{-\alpha \lceil R_l\rceil/2}}{1-\mathrm{e}^{-\alpha/2}}.   
   \]
Finally, using $\lceil R_l\rceil \ge 2l + (2/\alpha) \ln (2 \cdot 3^d C / (\lvert c_u \rvert (1-\mathrm{e}^{-\alpha/2}))$, we conclude Ineq.~\eqref{eq:desired} which ends the proof.
\end{proof}
\begin{lemma} \label{lem:nexp}
  Let $M,\alpha > 0$. Then
  $$ 
  n \ge \frac{8M^2}{\alpha^2} \quad \Rightarrow \quad n^M < \mathrm{e}^{\alpha n/2}. 
  $$ 
\end{lemma}
\begin{proof}
  If $n \ge 8M^2 / \alpha^2$ then
  $$ n \le \frac{\alpha^2 n^2}{8 M^2}. $$ 
  Since
  $$ \mathrm{e}^{\frac{\alpha n}{2M}} = 1 + \frac{\alpha n}{2M} + 
  \frac{\alpha^2 n^2}{8 M^2} + \dots > \frac{\alpha^2 n^2}{8 M^2 }, $$
  we conclude that $n \le \mathrm{e}^{\alpha n/(2M)}$, or, equivalently, $n^M \le \mathrm{e}^{\alpha n / 2}$.
\end{proof}
%
%% %%%%%%%%%%%%%%%%%%%%%%%%%%%%%%%%%%%%%%%%%%%%%%%%%%%%%%%%%%%%%%%%%%%%%%%%%%%%%%%%%%%
%--------------------------------------------------------------------------------
%
%          Abstract Wegner
% %--------------------------------------------------------------------------------
%%
%%%%%%%%%%%%%%%%%%%%%%%%%%%%%%%%%%%%%%%%%%%%%%%%%%%%%%%%%%%%%%%%%%%%%%%%%%%%%%%%%
%%
%

\section{Abstract Wegner estimates; proof of Theorem~\ref{theorem:wegner_c} and \ref{theorem:wegner_d}} \label{sec:abstract_wegner}
Recall that the space of functions of finite total variation $\BV (\RR)$ is 
the set of integrable functions $f : \RR \to \RR$ whose distributional derivatives are signed measures with finite variation, i.e.
\[
 \BV (\RR) \colonequals \{f : \RR \to \RR  \mid f \in L^1 (\RR),\ D f \ \text{is a signed measure},\ \lvert D f \rvert (\RR) < \infty\} .
\]
To say that a distributional derivative $Df$ of a function $f \in L^1_{\rm loc} (\RR)$ is a signed measure means that there exists a regular signed Borel measure $\nu$ on $\RR$ such that
\[
 \int_\RR \phi \drm \nu = - \int_\RR f \phi' \drm x
\]
for all $\phi \in C_{\rm c}^\infty (\RR)$. A norm on $\BV (\RR)$ is defined by $\lVert f \rVert_{\BV (\RR)} \colonequals \lVert f \rVert_{L^1 (\RR)} + \lVert f \rVert_{\rm Var}$, where 
\[
\lVert f \rVert_{\rm Var} \colonequals \lvert D f \rvert (\RR) = \sup \Bigl\{ \int_\RR f v' \drm x \colon v \in C_{\rm c}^\infty (\RR) , \ \lvert v \rvert \leq 1 \Bigr\} .
\]
Note that if $f \in W^{1,1} (\RR)$ then $f \in \BV (\RR)$. In particular, one has the equalities $\lVert f \rVert_{W^{1,1} (\RR)} = \lVert f \rVert_{\BV (\RR)}$ and $\lVert f \rVert_{L^1 (\RR)} = \lVert f \rVert_{\rm Var}$ provided the norms are well defined.
\par
In \cite{KostrykinV-06} an abstract Wegner estimate for the continuous model was established, which we will be able to use in our situation. Let us first fix some notation. For an open set $\Lambda \subset \RR^d$, $\tilde \Lambda$ is the set of lattice sites $j \in \ZZ^d$ such that the characteristic function of the cube $\Lambda_{1/2} (j)$ does not vanish identically on $\Lambda$. For $j \in \ZZ^d$ we denote by $\chi_j$ the characteristic function of the cube $\Lambda_{1/2} (j)$.
\begin{theorem}[\cite{KostrykinV-06}] \label{theorem:abstract1}
Let Assumption~\ref{ass:bv} be satisfied and assume there is $l_0 >0$ such that for arbitrary $l \geq l_0$ and every $j \in \tilde{\Lambda}_l$ there is a compactly supported sequence $t_{j,l} \in \ell^1 (\ZZ^d ; \RR)$ such that
\begin{equation*}
\sum_{k \in \ZZ^d} t_{j,l} (k) U(x-k) \geq \chi_j (x) \quad \text{for all} \quad x \in \Lambda_l .
\end{equation*}
Let further $I \colonequals [E_1,E_2]$ be an arbitrary interval. Then for any $l \geq l_0$
\begin{equation*}
\EE_\Gamma \{ \Tr P_I (H_\omega^{l})\} \leq C \euler^{E_2} \lVert \rho \rVert_{\rm Var} \lvert I \rvert \sum_{j \in \tilde{\Lambda}_l} \lVert t_{j,l} \rVert_{\ell^1 (\ZZ^d)} ,
\end{equation*}
where $C$ is a constant independent of $l$ and $I$ and $\Gamma = \bigcup_{j \in \tilde \Lambda_l}\supp t_{j,l}$.
\end{theorem}
In \cite{KostrykinV-06} this theorem was stated for compactly supported $U:\RR^d \to \RR$ and with the partial average $\EE_\Gamma (\ldots)$ replaced by the total average $\EE (\ldots)$. That the theorem holds in the slightly stronger form stated here can be seen by following the proof of Theorem~\ref{theorem:abstract2} treating the discrete case.
\begin{theorem}[Discrete analogue of \cite{KostrykinV-06}] \label{theorem:abstract2}
Let Assumption~\ref{ass:bv} be satisfied and assume there is $l_0 > 0$ such that for arbitrary $l \geq l_0$ and every $j \in \cd{l}$ there is a compactly supported sequence $t_{j,l} \in \ell^1 (\ZZ^d)$ such that
\begin{equation*}
\sum_{k \in \ZZ^d} t_{j,l} (k) u(x-k) \geq \delta_j (x) \quad \text{for all} \quad x \in \cd{l} .
\end{equation*}
Let further $I\colonequals[E_1,E_2]$ be an arbitrary interval. Then for any $l \geq l_0$ we have
\begin{equation*}
\EE_\Gamma \{ \Tr P_I (h_\omega^{l})\} \leq \frac{1}{2}\lVert \rho \rVert_{\rm Var} \lvert I \rvert \sum_{j \in \cd{l}} \lVert t_{j,l} \rVert_{\ell^1 (\ZZ^d)} ,
\end{equation*}
where $\Gamma = \bigcup_{j \in \cd{l}}\supp t_{j,l}$.
\end{theorem}
For the proof of Theorem~\ref{theorem:abstract2} we will use an estimate on averages of spectral projections of certain self-adjoint operators. More precisely, let $\mathcal{H}$ be a Hilbert space and consider the following operators on $\mathcal{H}$. Let $H$ be self-adjoint, $W$ symmetric and $H$-bounded, $J$ bounded and non-negative with $J^2 \leq W$, $H (\zeta) = H + \zeta W$ for $\zeta \in \RR$, and $P_I (H(\zeta))$ the corresponding spectral projection onto an interval $I \subset \RR$. Then, for any $g \in L^\infty (\RR) \cap L^1 (\RR)$, $\psi \in \mathcal{H}$ with $\lVert \psi \rVert = 1$ and bounded interval $I \subset \RR$,
\begin{equation} \label{eq:average_proj}
\int_\RR \bigl\langle \psi , J P_I (H(\zeta)) J \psi \bigr\rangle g(\zeta) \drm \zeta \leq \lVert g \rVert_\infty \lvert I \rvert .
\end{equation}
For a proof of Ineq.~\eqref{eq:average_proj} we refer to \cite{CombesH-94} where compactly supported $g$ is considered. The non-compactly supported case was first treated in \cite{FischerHLM-97}, see also \cite[Lemma~5.3.2]{Veselic-08} for a detailed proof.
\begin{proof}[Proof of Theorem~\ref{theorem:abstract2}]
In order to estimate the terms of the sum in the expectation 
\[
\EE_\Gamma \{\Tr P_I (h_\omega^l)\} = \sum_{j \in \cd{l}} \EE_\Gamma \{\lVert P_I(h_\omega^l) \delta_j \rVert^2\}
\]
we fix $l \geq l_0$ and $j \in \cd{l}$, and set $\Sigma = \supp t_{j,l} \subset \ZZ^d$ and $t = t_{j,l}$. 
Recall that
\[
 h_\omega^l = \pi_{\cd{l}}h_0 \iota_{\cd{l}} + \sum_{k \in \ZZ^d \setminus \Sigma} \omega_k u (\cdot - k) +  \sum_{k \in  \Sigma} \omega_k u (\cdot - k).
\]
We pick some $o \in \Sigma$ with $t (o) \not = 0$ and denote by $M$ the finite dimensional linear (and invertible) transformation $(\eta_k)_{k \in \Sigma} \mapsto (\omega_k)_{k \in \Sigma} = M (\eta_k)_{k \in \Sigma}$ defined as follows: $\omega_o = t (o) \eta_o$ and $\omega_k = t (k)  \eta_o + t(o) \eta_k$  for $k \in \Sigma \setminus \{o\}$. Note that $M$ is invertible and $\lvert \det M\rvert = \lvert t(o) \rvert^{\lvert \Sigma \rvert}$. With this transformation there holds for arbitrary fixed $(\omega_k)_{k \in \ZZ^d \setminus \Sigma}$
\begin{align*}
\int_{\RR^{\lvert \Sigma \rvert}} \bigl\lVert P_I (h_\omega^l) \delta_j \bigr\rVert^2 \prod_{k \in \Sigma} \rho (\omega_k) \drm \omega_k = 
\int_{\RR^{\lvert \Sigma \rvert}} \bigl\lVert P_I (h_{\omega , \eta}^l) \delta_j \bigr\rVert^2 k (\eta) \drm \eta , 
\end{align*}
where $\eta = (\eta_k)_{k \in \Sigma}$, $\drm \eta = \prod_{k \in \Sigma} \drm \eta_k$, 
\[
k (\eta) = \lvert t (o) \rvert^{\lvert \Sigma \rvert} \rho (t(o) \eta_o) \prod_{k \in \Sigma \setminus\{o\}} \rho (t(k) \eta_o + t(o) \eta_k),
\]
and
\[
h_{\omega , \eta}^l \colonequals  \pi_{\cd{l}} h_0 \iota_{\cd{l}} + \!\!\! \sum_{k \in \ZZ^d \setminus \Sigma} \!\!\! \omega_k u (\cdot - k) +  t(o) \!\!\! \sum_{k \in \Sigma \setminus \{o\}} \!\!\! \eta_k u(\cdot-k) + \eta_o \sum_{k \in \Sigma} t(k) u(\cdot-k) .
\]
We denote by $P_j : \ell^2 (\ZZ^d) \to \ell^2 (\ZZ^d)$ the orthogonal projection given by $P_j \phi = \phi (j) \delta_j$ and apply Ineq.~\eqref{eq:average_proj} with the choice $H = h_{\omega , \eta}^l - \eta_o \sum_{k \in \Sigma} t(k) u(\cdot-k)$, $W = \sum_{k \in \Sigma} t(k) u(\cdot-k)$, $\zeta = \eta_o$ and $J = P_j$. This gives by Lebesgue's theorem
\begin{equation} \label{eq:afteraverage}
\int_{\RR^{\lvert \Sigma \rvert}} \bigl\lVert P_I (h_\omega^l) \delta_j \bigr\rVert^2 \prod_{k \in \Sigma} \rho (\omega_k) \drm \omega_k  \leq \lvert I \rvert \int_{\RR^{\lvert \Sigma \rvert-1}} \sup_{\eta_o \in \RR} \lvert k(\eta) \rvert \prod_{k \in \Sigma\setminus \{o\}}\drm\eta_k .
\end{equation}
If $\rho \in W^{1,1} (\RR)$, we use $\sup_{\eta_o \in \RR} \lvert k(\eta) \rvert \leq \frac{1}{2} \int_\RR \lvert \partial_o k \rvert \drm \eta_o$. By the product rule we obtain for the partial derivative (while substituting back into original coordinates)
\[
 \frac{\partial}{\partial \eta_o} k(\eta) = \lvert t(o) \rvert^{\lvert \Sigma \rvert} \sum_{k \in \Sigma} t(k) \rho' (\omega_k) \prod_{j \in \Sigma \setminus \{k\}} \rho (\omega_j) .
\]
Hence, the right hand side of Ineq.~\eqref{eq:afteraverage} is bounded by $\frac{1}{2}\lvert I \rvert \lVert \rho' \rVert_{L^1 (\RR)} \sum_{k \in \Sigma} \lvert t(k) \rvert$. Since all the steps were independent of $j \in \cd{l}$, we in turn obtain the statement of the theorem in the case $\rho \in W^{1,1} (\RR)$.
We use the fact that for $\rho$ of bounded total variation and compact support
there is sequence $\rho_k \in C_{\rm c}^\infty (\RR)$, $k \in \NN$, such that $\lVert \rho_k \rVert_{L^1 (\RR)} = 1$ for all $k \in \NN$, $\lim_{k \to \infty} \lVert \rho_k \rVert_{\rm Var} = \lVert \rho \rVert_{\rm Var}$ and $\lim_{k \to \infty} \lVert \rho_k - \rho \rVert_{L^1 (\RR)} = 0$, see e.g.\ \cite{Ziemer-89} or Lemma~\ref{lemma:approx} below. Since $\lVert \rho_k \rVert_{\rm Var} = \lVert \rho_k' \rVert_{L^1 (\RR)}$ for $\rho_k \in C_{\rm c}^\infty (\RR)$, the same consideration as above gives 
\begin{equation}\label{eq:afteraverage2}
\int_{\RR^{\lvert \Sigma \rvert}} \bigl\lVert P_I (h_\omega^l) \delta_j \bigr\rVert^2 \prod_{i \in \Sigma} \rho_k (\omega_i) \drm \omega_i  \leq \frac{1}{2}\lvert I \rvert \lVert \rho_k \rVert_{\rm Var} \sum_{k \in \Sigma} \lvert t(k) \rvert
\end{equation}
for all $k \in \NN$. By a limiting argument, see \cite{KostrykinV-06} for details, one obtains Ineq.~\eqref{eq:afteraverage2} with $\rho_k$ replaced by $\rho$. This proves the theorem.
\end{proof}
\begin{lemma} \label{lemma:approx}
 Let $u: \RR \to \RR_0^+$ be a function of finite variation and bounded support. Assume additionally $\lVert u \rVert_{L^1 (\RR)} = 1$. Then there exists a sequence $u_k \in C_{\rm c}^\infty$, $k \in \NN$, such that $\lVert u_k \rVert_{L^1 (\RR)} = 1$ for all $k \in \NN$, 
\begin{equation} \label{eq:conv_d}
 \lim_{k \to \infty} \lVert u_k \rVert_{\rm Var} = \lVert u \rVert_{\rm Var} 
\end{equation}
and
\begin{equation} \label{eq:L1}
\lim_{k \to \infty} \lVert u_k - u \rVert_{L^1 (\RR)} = 0 .
\end{equation}
\end{lemma}
\begin{proof}
 Let $\phi \in C_{\rm c}^\infty (\RR)$ be non-negative with $\supp \phi \subset [-1,1]$ and $\lVert \phi \rVert_{L^1 (\RR)} = 1$. For $\epsilon > 0$ set $\phi_\epsilon : \RR \to \RR_0^+$, $\phi_\epsilon (x) = \epsilon^{-1} \phi (x/\epsilon)$. The function $\phi_\epsilon$ belongs to $C_{\rm c}^\infty (\RR)$ and fulfills $\lVert \phi_\epsilon \rVert_{L^1 (\RR)} = 1$. Now consider $u_\epsilon : \RR \to \RR_0^+$, 
\[
 u_\epsilon (x) = \int_\RR \phi_\epsilon (x-y) u(y) \drm y .
\]
Obviously, $u_\epsilon \in C_{\rm c}^\infty (\RR)$ and by Fubini's theorem $\lVert u_\epsilon \rVert_{L^1 (\RR)} = 1$. The proof of the relation \eqref{eq:L1} is due to Theorem~1.6.1 in \cite{Ziemer-89}. For the proof of the relation \eqref{eq:conv_d}, first note 
\begin{align*}
 \lVert u \rVert_{\rm Var} &= \lvert Du \rvert (\RR) = \sup \Bigl\{ \int_\RR u v' \drm x \colon v \in C_{\rm c}^\infty (\RR) , \ \lvert v \rvert \leq 1 \Bigr\} \\
&= \sup \Bigl\{ \lim_{\epsilon \searrow 0} \int_\RR u_\epsilon v' \drm x \colon v \in C_{\rm c}^\infty (\RR) , \ \lvert v \rvert \leq 1 \Bigr\} \\
& \leq \liminf_{\epsilon \searrow 0} \lvert D u_\epsilon \rvert (\RR) =  \liminf_{\epsilon \searrow 0} \lVert u_\epsilon \rVert_{\rm Var} ,
\end{align*}
since $u_\epsilon$ converges to $u$ in $L^1(\RR)$ and $v'$ is bounded. Let now $\psi \in C_{\rm c}^\infty (\RR)$ with $\vert\psi\rvert \leq 1$ and set $\psi_\epsilon = \phi_\epsilon * \psi$. Then we have by Fubini's theorem
\begin{align*}
 \lVert u \rVert_{\rm Var} &\geq \Bigl\lvert\int_\RR u \psi_\epsilon' \drm x \Bigr\rvert = \Bigl\lvert\int_\RR u (\psi * \phi_\epsilon)' \drm x \Bigr\rvert =  \Bigl\lvert\int_\RR  \psi'(u * \phi_\epsilon) \drm x \Bigr\rvert = \Bigl\lvert\int_\RR u_\epsilon \psi' \drm x \Bigr\rvert .
\end{align*}
Taking supremum over all such $\psi$ gives $\lVert u \rVert_{\rm Var} \geq \lVert u_\epsilon \rVert_{\rm Var}$. This proves the lemma.
\end{proof}
Assume Assumption~\ref{ass:exp}, i.e.\ that there are $C,\alpha \in (0,\infty)$ such that $\lvert u(x) \rvert \leq C \euler^{-\alpha \lVert x \rVert_1}$ for all $x \in \ZZ^d$. Let $I_0$ and $c_u \not = 0$ be as in Eq.~\eqref{eq:cF}. In Section~\ref{sec:transformation} we constructed for each $l >0$ a number $R_l >0$  such that
\begin{equation} \label{eq:vpos}
 \frac{2}{c_u} \sum_{k \in \cd{R_l}} k^{I_0} u(x-k) \geq 1 \quad \text{for all $x \in \cd{l}$} .
\end{equation}
This fact is proven in Proposition~\ref{prop2} and we apply it for the continuous model if $U$ is a generalized step-function with a exponential decaying convolution vector and for the discrete model with exponential decaying single-site potential to verify the hypothesis of Theorem~\ref{theorem:abstract1} and \ref{theorem:abstract2}. 
\begin{proof}[Proof of Theorem~\ref{theorem:wegner_d}]
Let $l_0 > 0$ be arbitrary. By Ineq.~\eqref{eq:vpos} (respectively Proposition~\ref{prop2}), the hypothesis of Theorem~\ref{theorem:abstract2} is satisfied with the choice $t_{j,l} \in \ell^1 (\ZZ^d)$ given by 
\[
t_{j , l} (k) = \begin{cases}
                 2 k^{I_0} / c_{u} & \text{if $k \in \cd{R_l}$}, \\
		0 & \text{else},
                \end{cases}
\]
for $l \geq l_0$ and $j \in \cd{l}$. It follows for all $l \geq l_0$ and $j \in \cd{l}$ that $\Gamma = \cd{R_l}$ and
\begin{align*}
 \sum_{j \in \cd{l}} \lVert t_{j,l} \rVert_{\ell^1 (\ZZ^d)} & \leq \frac{2}{\lvert c_{u} \rvert} (2l+1)^d \sum_{k \in \cd{R_l}} \lvert k^{I_0} \rvert \leq 
\frac{2}{\lvert c_{u} \rvert} (2l+1)^d (2R_l + 1)^d R_l^{\lVert I_0 \rVert_1} .
\end{align*}
Recall that by Proposition~\ref{prop2}, $R_l = \max \{2l + D , D'\} < 2l+D+D'$ with $D$ and $D'$ depending only on the single-site potential $u$. Hence there is a constant $C_{\rm W} > 0$ depending only on the single-site potential $u$ such that
\begin{equation*} \label{eq:volume}
 \sum_{j \in \cd{l}} \lVert t_{j,l} \rVert_{\ell^1 (\ZZ^d)} \leq
C_{\rm W} (2l+1)^{2d + \lVert I_0 \rVert_1}.
\end{equation*}
By Theorem~\ref{theorem:abstract2}, this completes the proof.
\end{proof}
\begin{proof}[Proof of Theorem~\ref{theorem:wegner_c}]
Recall that $U$ is a generalized step function and that $w\geq \kappa \chi_{[-1/2,1/2]^d}$ has compact support. 
Also recall, that for an open set $\Lambda \subset \RR^d$, $\tilde \Lambda$ is the set of lattice sites $j \in \ZZ^d$ such that the characteristic function of the cube $\Lambda_{1/2} (j)$ does not vanish identically on $\Lambda$. 
Recall $r = \sup \{ \lVert x \rVert_\infty : w(x) \not = 0\}$. Let $l_0 > 0$ be arbitrary and $t_{j,l} \in \ell^1 (\ZZ^d)$ given by 
\[
t_{j , l} (k) = \begin{cases}
                 2 k^{I_0} / (c_{u}\kappa) & \text{if $k \in \cd{R_{l+r}}$}, \\
		0 & \text{else} ,
                \end{cases}
\]
for $l \geq l_0$ and $j \in \tilde \Lambda_l$. It follows for all $l \geq l_0$ and $j \in \tilde \Lambda_l$ that $\Gamma = \cup_{j \in \tilde \Lambda_l} \supp t_{j,l} = \cd{R_{l+r}}$. By Ineq.~\eqref{eq:vpos} (respectively Proposition~\ref{prop2}) we have for all $l \geq l_0$, $j \in \tilde \Lambda_l$ and $x \in \Lambda_l$
\begin{align*}
 \sum_{k \in \ZZ^d} t_{j,l} (k) U(x-k) 
%&=  \sum_{k \in \Sigma_n} t_{j,n} (k) \sum_{i \in \ZZ^d} u (i) w (x-k-i) \\ 
&= \sum_{i \in \ZZ^d} w (x-i) \sum_{k \in \ZZ^d} t_{j,l} (k)  u (i-k) \\
&\geq \frac{1}{\kappa}\sum_{i \in \cd{l+r}} w(x-i) 
+ \!\!\! \sum_{i \in \ZZ^d \setminus \cd{l+r}} \!\!\!\!\! w (x-i) \sum_{k \in \ZZ^d} t_{j,l} (k)  u (i-k) \\
&= \frac{1}{\kappa} \sum_{i \in \cd{l+r}} w (x-i) \geq \chi_j (x) .
\end{align*}
Here we have used that $w (x-i) = 0$ for $x \in \Lambda_l$ and $i \not \in \cd{l+r}$. Hence the assumption of Theorem~\ref{theorem:abstract1} is satisfied. Analogous to the proof of Theorem~\ref{theorem:wegner_d} there is a constant $C_{\rm W}$ depending only on the single-site potential $U$ such that
\begin{equation*}
 \sum_{j \in \tilde \Lambda_l} \lVert t_{j,l} \rVert_{\ell^1 (\ZZ^d)} \leq
C_{\rm W} (2l+1)^{2d + \lVert I_0 \rVert_1}.
\end{equation*}
This completes the proof by using Theorem~\ref{theorem:abstract1}.
\end{proof}
%
%% %%%%%%%%%%%%%%%%%%%%%%%%%%%%%%%%%%%%%%%%%%%%%%%%%%%%%%%%%%%%%%%%%%%%%%%%%%%%%%%%%%%
%--------------------------------------------------------------------------------
%
%          Uniform Control of Resonances
% %--------------------------------------------------------------------------------
%%
%%%%%%%%%%%%%%%%%%%%%%%%%%%%%%%%%%%%%%%%%%%%%%%%%%%%%%%%%%%%%%%%%%%%%%%%%%%%%%%%%
%%
%
\section{Uniform control of resonances}\label{sec:resonances}
In order to carry out a proof of localization via multiscale analysis there is a need for an upper bound on the probability that there are resonances of two box Hamiltonians. Since the single-site potential may have unbounded support one needs a so-called uniform version of this estimate as proposed in \cite{KirschSS-98b}. Moreover, the uniform control of resonances can be used to verify the initial lentgh scale estimate in the large disorder regime.
\par
If Assumptions~\ref{ass:exp} and \ref{ass:bv} are satisfied, the Wegner estimate from Theorem~\ref{theorem:wegner_d} tells us for all $E \in \RR$, $l >0$ and $\epsilon > 0$ that
\[
 \sup_{\omega_k \in \RR, \ k \in \ZZ^d \setminus \cd{R_l}} \EE_{\cd{R_l}} \bigl(\Tr P_{[E-\epsilon , E + \epsilon]} (h_\omega^l)\bigr) \leq C_{\rm W} \lVert \rho \rVert_{\rm Var} 2\epsilon (2l+1)^{2d + \lVert I_0 \rVert_1} ,
\]
where $R_l$ is given in Eq.~\eqref{eq:RLrel}. A similar estimate holds for our alloy-type model in $L^2 (\RR^d)$, cf.\ Theorem~\ref{theorem:wegner_c}.
\par
In order to formulate the uniform control of resonances let $x,y \in \ZZ^d$ and $l_1,l_2 >0$ be such that the cubes 
\[
\cc_1 \colonequals \cd{l_1}\rkl{x}, \quad \cc_{1}^+ \colonequals \cd{4l_1}\rkl{x}, \quad \cc_{2}\colonequals \cd{l_2}\rkl{y} \quad \text{and} \quad \cc_2^+ \colonequals \cd{4l_2}\rkl{y} 
\]
satisfy $\cc_{1}^+ \cap  \cc_{2}^+ = \emptyset$. We
  define the map $\Pi_{\cd{}} : \Omega \to \Omega_{\cd{}}$ by
  $(\Pi_{\cd{}} \omega)_j = \omega_j$ for $j \in \cd{}$. For $\omega \in \Omega$ we set $\omega_{i}=\Pi_{\cd{i}^+} \omega, \ i\in \{1,2\}$, and define the uniform distance by
\begin{equation} \label{eq:unif_dist}
 \tilde{d}\bigl(\sigma(h^{\cc_1}_{\omega}),\sigma(h^{\cd{2}}_{\omega})\bigr) \colonequals \inf_{\substack{\omega_1^{\bot} \in \, \Omega_{\ZZ^d\setminus \cd{1}^+},\\ \omega_2^{\bot} \in \, \Omega_{\ZZ^d\setminus \cd{2}^+}}} d\Bigl(\sigma\bigl(h^{\cd{1}}_{(\omega_1,\omega_1^{\bot})}\bigr),\sigma\bigl(h^{\cd{2}}_{(\omega_2,\omega_2^{\bot})}\bigr)\Bigr).
\end{equation}
We will use an analogue notation for the continuous operator in $L^2$. For $\epsilon > 0$ we define the event 
\begin{equation} \label{eq:eventA}
  A(\cd{1},\cd{2} , \epsilon)\colonequals\left\{\omega \in \Omega: \tilde{d}(\sigma(h^{\cd{1}}_{\omega}),\sigma(h^{\cd{2}}_{\omega})) < \epsilon \right\} .
\end{equation}
In analogy we define for the continuum setting for $\epsilon > 0$ and $J \subset \RR$ the event
\begin{equation*} 
  A_{J}(\Lambda_1,\Lambda_2 , \epsilon)\colonequals\left\{\omega \in \Omega: \tilde{d}(\sigma_J(H^{\Lambda_1}_{\omega}),\sigma_J(H^{\Lambda_2}_{\omega})) < \epsilon \right\} .
\end{equation*}
Here $\Lambda_1$ and $\Lambda_2$ are cubes with centers $x,y \in \RR^d$ and side lengths $l_1$ and $l_2$, with the property that $\Lambda_1^+ = \Lambda_{4l_1} (x)$ and $\Lambda_2^+ = \Lambda_{4l_2} (y)$ are disjoint. The notation $\sigma_J(H^{\Lambda_i}_{\omega})$ is for $J \subset \RR$ defined by
\[
 \sigma_J (H^{\Lambda_i}_{\omega}) = \sigma (H^{\Lambda_i}_{\omega}) \cap J , \quad i \in \{1,2\} .
\]
For the discrete model we have: 
\begin{proposition} \label{prop:unif_wegner_type}
Let Assumptions~\ref{ass:exp} and \ref{ass:bv} be satisfied. Then there are constants $C_1 =C_1(u,\rho) >0$, $C_2 = C_2 (u,\rho)$ and $l_0 = l_0 (u)$, such that if $l_2 \geq l_0$ we have for all $\epsilon > 0$ the bound
\begin{equation*}
 \mathbb{P}\rkl{A(\cd{1},\cd{2} , \epsilon)} 
\leq C_1 (2\max\{l_1 , l_2 \} + 1)^{3d+ \lVert I_0 \rVert_1} \bigl[\epsilon + C_2 \euler^{-3 \min\{l_1 , l_2\} \alpha / 2} \bigr].
\end{equation*}
\end{proposition}
For the proof we need some preparatory estimate.
\begin{lemma} \label{lemma:eig_perturb}
Let Assumption~\ref{ass:exp} be satisfied. Then there exists a constant $\hat C = \hat C(C,\alpha , d)$ such that for all $l,l' \geq 0$ and all $x \in \cd{l}$ there holds
\[
\sum_{\lVert k \rVert_\infty > l+l'} \lvert u(x-k) \rvert \leq \hat C \mathrm{e}^{-\alpha l' / 2}  .
\]
\end{lemma}
\begin{proof}
Since $\lVert x-k \rVert_\infty > l'$ for all $k$ with $\lVert k \rVert_\infty > l+l'$, we have the estimate
\[
S\colonequals \sum_{\lVert k \rVert_\infty > l+l'} \lvert u(x-k) \rvert \leq \sum_{\lVert x-k \rVert_\infty > l'} \lvert u(x-k) \rvert
= \sum_{\lVert k \rVert_\infty > l'} \lvert u(k) \rvert .
\]
We use the exponential decay condition on $u$, $\lVert k \rVert_1 \geq \lVert k \rVert_\infty$, and obtain
\begin{align*}
S &\leq C \sum_{\lVert k \rVert_\infty > l'} \mathrm{e}^{-\alpha \lVert k \rVert_1} 
\leq C \mathrm{e}^{-\alpha l' / 2} \sum_{\lVert k \rVert_\infty > l'} \mathrm{e}^{-\alpha \lVert k \rVert_\infty / 2} \\[1ex]
&\leq C \mathrm{e}^{-\alpha l' / 2} \sum_{k \in \mathbb{Z}^d} \mathrm{e}^{-\alpha \lVert k \rVert_\infty / 2} .
\end{align*}
The sum itself assumes a finite value $C_{\alpha,d} > 0$, depending only on $\alpha$ and $d$, so we have for all $l,l' >0$ that $S \leq C C_{\alpha,d} \mathrm{e}^{-\alpha l'/2}$.
\end{proof}
\begin{proof}[Proof of Proposition~\ref{prop:unif_wegner_type}]
Let $i \in \{1,2\}$ and $\omega,\omega' \in \Omega$ with $\Pi_{\cd{i}^+} \omega = \Pi_{\cd{i}^+} \omega'$. By Lemma~\ref{lemma:eig_perturb} we have for all 
$x \in \cd{i}$ that
\begin{align}  \label{eq:perturb}
\lvert v_\omega (x) - v_{\omega'} (x) \rvert 
&\leq \sum_{k \not \in \cd{i}^+} \lvert u(x-k) \rvert \lvert \omega_k - \omega_k' \rvert  \leq C_{u,\rho} \euler^{-3l_i \alpha / 2} =: \delta_i .
\end{align}
Here $C_{u,\rho}$ is a constant depending only on the single-site potential $u$ and the density $\rho$. From this fact there follows that the eigenvalues of $h_\omega^{\cd{i}}$ move at most by $\delta_i$ if the configuration changes from $\omega$ to $\omega'$. 
\par
It will be convenient to take a special choice of $\omega'$, namely one where all coupling constants outside a finite box are set equal to zero. More precisely, set
\begin{align*}
h_\omega^1 \colonequals h_{\hat \omega}^{\cd{1}}, \ \text{where $\hat\omega_k =
\omega_k \mathbf{1}_{\cd{1}^+}(k)$}, \\
h_\omega^2 \colonequals h_{\hat \omega}^{\cd{2}}, \ \text{where $\hat\omega_k =
\omega_k \mathbf{1}_{\cd{2}^+}(k)$} .
\end{align*}
Hence,
\begin{align}
& \tilde d ( \sigma (h_\omega^{\cd{1}}) , \sigma (h_\omega^{\cd{2}}))) < \epsilon , \nonumber \\
\Rightarrow \quad & d ( \sigma (h_\omega^1) , \sigma (h_\omega^2)) < \epsilon + \delta_1 + \delta_2 , \nonumber \\
\Rightarrow \quad & \exists \, E \in \sigma (h_\omega^1) : d(E,\sigma (h_\omega^2)) < \epsilon + \delta_1 + \delta_2 .  \label{eq:implication}
\end{align}
The set of $\omega \in \Omega$ where \eqref{eq:implication} holds will be denoted by $B$. Thus
 \[
 \PP (A(\cd{1} , \cd{2} , \epsilon)) \leq \PP (B) = \EE \bigl[ \EE_{\cd{2}^+} \bigl( \mathbf{1}_{B} \bigr) \bigr].
\]
Next we provide a uniform upper bound on the random variable $\EE_{\cd{2}^+}( \mathbf{1}_{B} )$. Using the pointwise estimate
\begin{multline*}
 \mathbf{1}_{\displaystyle \{\omega \in \Omega \colon \exists \, E \in \sigma (h_\omega^1) : d(E,\sigma (h_\omega^2)) < \epsilon + \delta_1 + \delta_2\}} (\omega) 
\\ \leq
\sum_{E \in \sigma (h_\omega^1)} \mathbf{1}_{\displaystyle \{\omega \in \Omega \colon d(E,\sigma (h_\omega^2)) < \epsilon + \delta_1 + \delta_2\}} (\omega)
\end{multline*}
and Chebychev's inequality we obtain
\begin{align*}
 \EE_{\cd{2}^+} \bigl( \mathbf{1}_{B} \bigr) &\leq \sum_{E \in \sigma (h_\omega^1)} \EE_{\cd{2}^+} \Bigl( \mathbf{1}_{\displaystyle \{\omega \in \Omega \colon d(E,\sigma (h_\omega^2)) < \epsilon + \delta_1 + \delta_2\}} \Bigr) \\
& \leq \sum_{E \in \sigma (h_\omega^1)} \EE_{\cd{2}^+} \Bigl( \Tr P_I (h_\omega^2) \Bigr) ,
\end{align*}
where $I = [E-\epsilon-\delta_1-\delta_2 , E+\epsilon+\delta_1+\delta_2]$.
Here we used that the set $\sigma (h_\omega^1)$ is independent of the random variables $\omega_k$, $k \in \cd{2}^+$. If $\cd{2}^+ \supset \cd{R_{l_2}}$, 
i.e., 
\[
 4l_2 \geq R_{l_2} = \max \left\{ 2l_2 + \frac{2}{\alpha} \ln \frac{2\, 3^d\, C}
  {\lvert c_u \rvert (1-\mathrm{e}^{-\alpha/2})}, \frac{8 (d+\lVert I_0 \rVert_1)^2}{\alpha^2} \right\} ,
\]
we can apply Theorem~\ref{theorem:wegner_d} and obtain 
\[
 \EE_{\cd{2}^+} \bigl( \mathbf{1}_{B} \bigr) \leq (2l_1 + 1)^d C_{\rm W} \lVert \rho \rVert_{\rm Var} 2(\epsilon + \delta_1 + \delta_2) (2l_2 + 1)^{2d + \lVert I_0 \rVert_1} 
\]
uniformly in $\omega_k$, $k \in \ZZ^d \setminus \cd{2}^+$. 
Thus 
\[
 \EE \bigl[ \EE_{\cd{2}^+} \bigl( \mathbf{1}_{B} \bigr) \bigr]\leq (2l_1 + 1)^d C_{\rm W} \lVert \rho \rVert_{\rm Var} 2(\epsilon + \delta_1 + \delta_2) (2l_2 + 1)^{2d + \lVert I_0 \rVert_1}. \qedhere
 \]
\end{proof}
An analogue of Proposition~\ref{prop:unif_wegner_type} holds true for the continuous alloy-type model with a single-site potential of generalized step function form. 
\begin{proposition} \label{prop:unif_wegner_type_cont}
Assume that $U$ is a generalized step function and that Assumptions~\ref{ass:exp} and \ref{ass:bv} are satisfied. Let further $J \subset\RR$ be a bounded interval. Then there are constants $C_3 =C_4(d) >0$, $C_4 = C_4 (U,\rho)$, $C_5 = C_5 (U,\rho)$, $l_0^* = l_0^* (U)$ and $l_1^* = l_1^* (U)$, such that if $l_2 \geq l_0^*$ and $l_1,l_2 \geq l_1^*$ we have for all $\epsilon > 0$ the bound
\begin{multline*}
 \mathbb{P}\rkl{A_{J}(\Lambda_1,\Lambda_2 , \epsilon)} 
\leq C_3 (\max\{\sup J , -C_5\} + C_5)^{d/2} \euler^{\lvert \sup J \rvert}  \\ \times (2\max\{l_1 , l_2 \} + 1)^{3d+\lVert I_0 \rVert_1}  \bigl[\epsilon + C_4 (2 \max \{l_1 , l_2\} + 1)^d \euler^{-3 \min\{l_1 , l_2\} \alpha / 2} \bigr].
\end{multline*}
\end{proposition}
\begin{proof}
Let $i \in \{1,2\}$ and $\omega,\omega' \in \Omega$ with $\Pi_{\cd{i}^+} \omega = \Pi_{\cd{i}^+} \omega'$. Then we have for
$x \in \Lambda_i$ that
\begin{align*}  %\label{eq:perturb}
\lvert V_\omega (x) - V_{\omega'} (x) \rvert 
&\leq \sum_{k \not \in \cd{i}^+} \lvert U(x-k) \rvert \lvert \omega_k - \omega_k' \rvert \\
&\leq C_\rho \sum_{j \in \ZZ^d} \lvert w (x-j) \rvert \sum_{k \not \in \cd{i}^+} \lvert u(j-k) \rvert 
\end{align*}
with some constant $C_\rho$ depending only on the probability density $\rho$. Recall that $r = \sup \{\rVert x \rVert_\infty \colon w(x) \not = 0\}$. Since $x \in \Lambda_i$ we have $w(x-j) = 0$ for $j \not \in \cd{l_i + r}$. If we assume that $l_i \geq r/3 =: l_1^*(U)$, $i \in \{1,2\}$, then using Lemma~\ref{lemma:eig_perturb} with $l = l_i + r$ and $l' = 3l_i - r$ we have
\[
 \lvert V_\omega (x) - V_{\omega'} (x) \rvert \leq C_\rho \sum_{j \in \cd{l_i + r}} \lvert w(x-j) \rvert \hat C \euler^{-\alpha (3 l_i - r) / 2} .
\]
For the moment we assume $w \in L^\infty (\RR^d)$. Then we have
\[
 \lvert V_\omega (x) - V_{\omega'} (x) \rvert \leq C_{U ,\rho}  (2l_i + 1)^d \euler^{-\alpha 3 l_i / 2} =: \delta_i
\]
with some constant $C_{U ,\rho}$ depending on the single-site potential $U$ and the density $\rho$.
From this fact there follows that the eigenvalues of $H_\omega^{\Lambda_i}$ move at most by $\delta_i$ if the configuration changes from $\omega$ to $\omega'$. For general $w$ as in Definition~\ref{def:step} the same fact holds by using Proposition~2.1 in \cite{KirschM-82}.
\par
It will be convenient to take a special choice of $\omega'$, namely one where all coupling constants outside a finite box are set equal to zero. More precisely, set
\begin{align*}
H_\omega^1 \colonequals H_{\hat \omega}^{\Lambda_1}, \ \text{where $\hat\omega_k =
\omega_k \mathbf{1}_{\cd{1}^+}(k)$}, \\
H_\omega^2 \colonequals H_{\hat \omega}^{\Lambda_2}, \ \text{where $\hat\omega_k =
\omega_k \mathbf{1}_{\cd{2}^+}(k)$} .
\end{align*}
Hence,
\begin{align}
& \tilde d ( \sigma_J (H_\omega^{\Lambda_1} , \sigma_J (H_\omega^{\Lambda_2}))) < \epsilon , \nonumber \\
\Rightarrow \quad & d ( \sigma_J (H_\omega^1) , \sigma_J (H_\omega^2)) < \epsilon + \delta_1 + \delta_2 , \nonumber \\
\Rightarrow \quad & \exists \, E \in \sigma_J (H_\omega^1) : d(E,\sigma_J (H_\omega^2)) < \epsilon + \delta_1 + \delta_2 .  \label{eq:implication2}
\end{align}
The set of $\omega \in \Omega$ where \eqref{eq:implication2} holds will be denoted by $B$. Thus
 \[
 \PP (A_J(\Lambda_1 , \Lambda_2 , \epsilon)) \leq \PP (B) = \EE \bigl[ \EE_{\cd{2}^+} \bigl( \mathbf{1}_{B} \bigr) \bigr].
\]
Next we provide a uniform upper bound on the random variable $\EE_{\cd{2}^+}( \mathbf{1}_{B} )$. Using the pointwise estimate
\begin{multline*}
 \mathbf{1}_{\displaystyle \{\omega \in \Omega \colon \exists \, E \in \sigma_J (H_\omega^1) : d(E,\sigma_J (H_\omega^2)) < \epsilon + \delta_1 + \delta_2\}} (\omega) 
\\ \leq
\sum_{E \in \sigma_J (H_\omega^1)} \mathbf{1}_{\displaystyle \{\omega \in \Omega \colon d(E,\sigma_J (H_\omega^2)) < \epsilon + \delta_1 + \delta_2\}} (\omega)
\end{multline*}
and Chebychev's inequality we obtain
\begin{align*}
 \EE_{\cd{2}^+} \bigl( \mathbf{1}_{B} \bigr) &\leq \sum_{E \in \sigma_J (H_\omega^1)} \EE_{\cd{2}^+} \Bigl( \mathbf{1}_{\displaystyle \{\omega \in \Omega \colon d(E,\sigma_J (H_\omega^2)) < \epsilon + \delta_1 + \delta_2\}} \Bigr) \\
& \leq \sum_{E \in \sigma_J (H_\omega^1)} \EE_{\cd{2}^+} \Bigl( \Tr P_I (H_\omega^2) \Bigr) ,
\end{align*}
where $I = [E-\epsilon-\delta_1-\delta_2 , E+\epsilon+\delta_1+\delta_2]$.
Here we used that the set $\sigma_J (H_\omega^1)$ is independent of the random variables $\omega_k$, $k \in \cd{2}^+$. If $\cd{2}^+ \supset \cd{R_{l_2+r}}$, i.e.\ $4l_2 \geq R_{l_2 + r}$ or $l_2 \geq l_0^* (U)$ with an appropriate chosen $l_0^* (U)$, we can apply Theorem~\ref{theorem:wegner_c} and obtain 
\[
 \EE_{\cd{2}^+} \bigl( \mathbf{1}_{\tilde B} \bigr) \leq \sum_{E \in \sigma_J (H_\omega^1)} \euler^{\lvert \sup J \rvert}C_{\rm W} \lVert \rho \rVert_{\rm Var} 2(\epsilon + \delta_1 + \delta_2) (2l_2 + 1)^{2d + \lVert I_0 \rVert_1} 
\]
uniformly in $\omega_k$, $k \in \ZZ^d \setminus \cd{2}^+$. Recall that the number of eigenvalues $\lvert \sigma_J (H_\omega^1) \rvert$ satisfies the bound $\lvert \sigma_J (H_\omega^1) \rvert \leq C \lvert \Lambda \rvert$ for all $\omega \in \Omega$. Here $C$ depends on the space dimension $d$, $\sup J$, 
the single-site potential $U$ and the measure $\mu$. This bound can be obtained by using the perturbation bound \eqref{eq:loc-lp-norm}, 
see e.g.\ Proposition~2.1 of \cite{KirschM-82}, and the well known Weyl bound for the number of eigenvalues of the Dirichlet Laplacian less then $\lambda$, see e.g.\ \cite{ReedS-78d}. Since the right hand side is now independent of $\omega$ we obtain the statement of the proposition.
\end{proof}
%
%% %%%%%%%%%%%%%%%%%%%%%%%%%%%%%%%%%%%%%%%%%%%%%%%%%%%%%%%%%%%%%%%%%%%%%%%%%%%%%%%%%%%
%--------------------------------------------------------------------------------
%
%          MSA
% %--------------------------------------------------------------------------------
%%
%%%%%%%%%%%%%%%%%%%%%%%%%%%%%%%%%%%%%%%%%%%%%%%%%%%%%%%%%%%%%%%%%%%%%%%%%%%%%%%%%
%%
%

\section{Localization via multiscale analysis (discrete model)} \label{sec:loc_discrete}
\subsection{Basic notation}
For $\cc \subset \ZZ^d$ we denote by $\partial^{\rm i} \cc = \{k \in \cc : \# \{j \in \cc : \lVert k-j \rVert_1 = 1\} < 2d\}$ the interior boundary of $\cc$ and by $\partial^{\rm o} \cc = \partial^{\rm i} \cc^{\rm c}$ the exterior boundary of $\cc$. Here $\cc^{\rm c} = \ZZ^d \setminus \cc$ denotes the complement of $\cc$. Moreover, we define the bond-boundary $\partial^{\rm b} \cc$ of
$\cc$ as
\[
\partial^{\rm b} \cc = \left\{(u,u') \in \ZZ^d \times \ZZ^d :
u\in \cc,\ u'\in \ZZ^d \setminus \cc,\ \text{and} \ \lVert
u-u' \rVert_1 = 1\right\}\,.
\]
For $\cc \subset \ZZ^d$, $u,w \in \cc$ and $E \in \CC \setminus \sigma (h_\omega^\cc)$ we denote the Green's function by
\[
 G_\omega^\cc (E;u,w) \colonequals \langle \delta_u , (h_\omega^\cc - E)^{-1} \delta_w \rangle .
\]
For the following definitions let $u \in \ZZ^d$ and $l > 0$.
\begin{definition}\label{def:mereg}
Let $m>0$ and $E \in \RR$. A cube $\cd{l} (u)$ is called \emph{$(m,E)$-regular} (for fixed $\omega \in \Omega$), if $E \not \in \sigma (h_\omega^{\cd{l} (u)})$ and
\[
 \sup_{w \in \partial^{\rm i} \cd{l} (u)} \lvert G_\omega^{\cd{l} (u)} (E;u,w) \rvert \leq \euler^{-ml} .
\]
Otherwise the cube $\cd{l} (u)$ is called \emph{$(m,E)$-singular}.
\end{definition}
\begin{definition}
Let $G_{l} (u,m,E) = \{\omega \in \Omega \mid \cd{l} (u) \ \text{is $(m,E)$-regular}\}$. A cube $\cd{l} (u)$ is called \emph{uniformly $(m,E)$-regular} (for fixed $\omega \in \Omega$), if $\omega' \in G_l (u,m,E)$ for all $\omega'$ with $\Pi_{\cd{4l} (u)} \omega' = \Pi_{\cd{4l} (u)} \omega$.
\end{definition}
\begin{definition}
\label{def:E-NR}
 Let $\omega \in \Omega,\ \zeta > 0$ and $E \in \RR$. We call a cube $\cd{l}(u)$ non-resonant for $\omega$ at energy $E$, $E$-NR for short, if 
\[
d\Bigl(E,\sigma \bigl(h_{\omega}^{\cd{l}(u)} \bigr) \Bigr) \geq \frac{1}{2} l^{-\zeta}, 
\]
or equivalently
\[
 \lVert G_\omega^{\cd{l}(u)}(E)\rVert \leq 2 l^\zeta .
\]
\end{definition}
It is convenient to introduce in accordance with \eqref{eq:unif_dist} also the following uniform distance
\begin{equation*} \label{eq:inif_dist2}
 \tilde{d}\bigl(E,\sigma(h^{\cd{l} (u)}_{\omega})\bigr) \colonequals \inf_{\omega_1^{\bot} \in \, \Omega_{\ZZ^d\setminus \cd{4l} (u)}} 
d\Bigl(E,\sigma\bigl(h^{\cd{l}(u)}_{(\omega_1 , \omega_1^\bot )}\bigr)\Bigr),
\end{equation*}
where $\omega_1 = \Pi_{\cd{4l}(u)} \omega$. 

\begin{definition}
Let $\omega \in \Omega,\ \zeta > 0$ and $E \in \RR$.
We call a cube $\cd{l}(u)$ uniformly non-resonant for $\omega \in \Omega$ at energy $E$, uniformly $E$-NR for short, if
\[
\tilde d \biggl(E,\sigma\Bigl(h_{\omega}^{\cd{l}(u)} \Bigr) \biggr) \geq \frac{1}{2} l^{-\zeta} .
\]
\end{definition}
\subsection{Induction step of the multiscale analysis} \label{sec:induction}
In this section we will carry out the induction step of the multiscale analysis.
We rely on the ideas and presentation from \cite{DreifusK-89} and \cite{KirschSS-98b}. 
For an interval $I \subset \RR$, $m > 0$, $l \in \NN$ and $z_1,z_2 \in \ZZ^d$ we define the event
\begin{multline*}
 B_l(z_1,z_2,m,I) \colonequals\{\omega \in \Omega: \\ \exists E \in I:\cd{l}(z_1) \text{ and } \cd{l}(z_2) \text{ are not uniformly }(m,E)\text{-regular}\} .
\end{multline*}
Note that the event $B_l(z_1,z_2,m,I)$ is independent of the coordinates $\omega_k$, $k \in \ZZ^d \setminus (\cd{4l}(z_1) \cup \cd{4l}(z_2))$. Therefore, $B_l(z_1,z_2,m,I)$ is a $\cd{4l}(z_1)\cup\cd{4l}(z_2)$ cylinder set. With this in mind, we obtain the independence of the events
\begin{equation} \label{eq:independent}
 B_l(z_1,z_2,m,I) \quad \text{and} \quad  B_l(z_3,z_4,m,I)
\end{equation}
provided $(\cd{4l}(z_1) \cup \cd{4l}(z_2)) \cap (\cd{4l}(z_3) \cup \cd{4l}(z_4)) = \emptyset$. 
Let now additionally $\xi >0$ be given. We say that the estimate  $G(I,l,m,\xi)$ is satisfied, if for all $z_1, z_2 \in \ZZ^d$ with $\lVert z_1-z_2 \rVert_{\infty} > 8l$ there holds
\[
 \PP(B_l (z_1,z_2,m,I)^{\rm c})\geq 1-l^{-2\xi} .
\]
Note that, if $G(I,l,m,\xi)$ is satisfied, then $\PP(B_l(x,y,m,I) )\leq l^{-2\xi}$ for all $x,y \in \ZZ^d$ with $\lVert x-y \rVert_\infty > 8l$.
\begin{theorem} \label{thm:induction}
Let Assumptions~\ref{ass:exp} and \ref{ass:bv} be satisfied, fix $\xi >2d$, $\kappa \in (1,2\xi /(\xi+2d)$ and $\beta \in (2-\kappa,1)$.
Then there exists $l^* = l^*(d,\xi,\kappa,\beta,u,\rho) \geq 1$
such that the following implication holds: 
\par
If for $l\geq l^*$ and $m > l^{\beta-1}$ the estimate $G(I,l,m,\xi)$ is satisfied, 
then also $G(I,L,m_L,\xi)$ is true where $L=l^\kappa$ and $m_L > 0$ satisfies
\begin{equation} \label{eq:m_L}
 m > m_L > m\bigl(1-L^{-(1-\beta)/\kappa}\bigr)-L^{-(1-\beta)/\kappa} > L^{\beta-1} .
\end{equation}
\end{theorem}
\begin{proof}
We define for $z \in \ZZ^d$ the event
\[
\begin{split}
 \Omega_{\rm G}(z)\colonequals\{& \omega \in \Omega: \ \forall E \in I \text{ there are no } 4 \text{ cubes }\cd{l}(z_1),\dots,\cd{l}(z_4) \subset \cd{L}(z)\\& \text{with } d(z_i,z_j)>8l \text{ for } i \neq j \text{ and }\\& \cd{l}(z_i)\text{ is not uniformly } (m,E)\text{-regular, for }i=1,\dots,4\}.
\end{split} 
\]
For $\omega \not\in \Omega_{\rm G}(z)$ there is some $E \in I$ for which 4 bad (= not uniformly $(m,E)$-regular) cubes $\cd{l}(z_i)$ with distance between their centers bigger than $8l$ exist. With $S\colonequals\{(z_1,\dots,z_4) \in \cd{L-\lfloor l \rfloor}^4(z): d( z_i,  z_j) > 8l, \text{ for } \ i \neq j\}$ we can write 
\[
\begin{split}
 \Omega^{\rm c}_{\rm G}(z) \subset \bigcup\limits_{(z_1,\dots,z_4)\in S} \bigl(&B_l(z_1,z_2,m,I) \cap B_l(z_3,z_4,m, I)\bigr).
\end{split}
\]
Since $B_l(z_1,z_2,m,I)$ and $B_l(z_3,z_4,m,I)$ are independent by \eqref{eq:independent}, 
we get with $\# S \leq (2L)^{4d}$
\[
 \PP(\Omega_{\rm G}^{\rm c}(z)) \leq \frac{(2L)^{4d}}{l^{4\xi}} = 2^{4d} L^{4(d-\frac{\xi}{\kappa})}.
\]
Since $4(d-\frac{\xi}{\kappa} )< -2\xi$ we can find $l_1^*=l_1^*(d,\xi,\kappa)$ such that for $l\geq l_1^*$
\begin{equation*}
 \PP(\Omega^{\rm c}_{\rm G}(z)) \leq \frac{1}{3}L^{-2\xi}.
\end{equation*}
Note that the event $\Omega_{\rm G}(z)$ merely means that there are at most 3 bad cubes 
(with sidelength $2\lfloor l \rfloor$) with sufficiently separated midpoints but, 
indeed, there might be many more bad cubes with midpoints in their neighborhoods.
 For a reason we will see later, we want to cover all bad cubes by bigger cubes, 
so called ``containers'', such that all sites outside these containers 
(but their centers still in the big cube $\cd{L}(z)$) are midpoints of uniformly $(m,E)$-regular cubes of sidelength $2\lfloor l \rfloor$.
 We use $\mathfrak{Ct}_i$ as notation for the $i$th container with sidelength $2 \lfloor l_i \rfloor$.
 The later application requires that the containers do not \emph{touch}. 
We say that two cubes $\cd{}, \cd{}'$ touch if they intersect or, 
in the case that they are disjoint, $\exists w \in \cd{}, w' \in \cd{}'$ 
with $\lVert w-w' \rVert_1 = 1$.
\par
Note that the event $\Omega_{\rm G}(z)$ ensures that we can find $z_1,z_2,z_3 \in \cd{L}(z)$ with $d(z_i,z_j) > 8l$ and
\[
 \bigl \{u \in \cd{L-\lfloor l\rfloor}(z) \colon \cd{l}(u) \text{ is bad} \bigr\} \subset \bigcup\limits_{j=1}^{3}\cd{8 l }(z_i).
\]
If none of the above $8l$-cubes touch, the $\cd{8l}(z_i)$ are already our containers. 
If two of the $8l$-cubes touch, we replace them by one $\cd{16l+1}(w)$ cube containing both of them. 
If $\cd{16l+1}(w)$ does not touch the remaining $8l$-cube, we choose these two cubes as our containers. 
Otherwise we choose only one container $\cd{24l+2}(u)$ containing all three original $8l$-cubes. 
By the construction we obtain $N\leq 3$ containers $\mathfrak{Ct}_i$, $i=1, \ldots,N$  
which do not touch and whose sidelengths $2\lfloor l_i \rfloor$, 
with $l_i \in \mathfrak{L}\colonequals\{8l,16l+1,24l+2\}$, satisfy the relation
\begin{equation} \label{eq:cont_length}
48l+3 \leq\sum_{i=1}^N (2l_{i} +1)\leq 48l+5.
\end{equation}
\par
Now we turn to the control of probabilities of resonances using the Wegner estimate. 
For fixed $x,y \in \mathbb{Z}^d$ with $\lVert x-y \rVert_{\infty} > 8L$ and $\zeta = \kappa(5d+ \lVert I_0 \rVert_1 + 2\xi) + 1$, we define
\begin{multline} 
\Omega_{\rm W}(x,y)
\colonequals \bigl\{\omega \in \Omega: 
\exists 
\cd{l_x}(x') \subset \cd{L}(x), \cd{l_y}(y') \subset \cd{L}(y) \text{ with} \\ l_x,l_y \in \mathfrak L\cup\{L\} 
\text{ and } \tilde{d}\bigl(\sigma(h^{\cd{l_x}(x')}_{\omega}),\sigma(h^{\cd{l_y}(y')}_{\omega})\bigr) < \min\{l_x , l_y\}^{-\zeta} \bigr\}.
\end{multline}
We want to get an upper bound for the probability of $\Omega_{\rm W}(x,y)$. By subadditivity we have
\[
 \PP(\Omega_{\rm W}(x,y)) \leq \sum\limits_{l_x , l_y , x' , y'} \PP \bigl( A(\cd{l_x}(x'),\cd{l_y}(y') , \min\{l_x , l_y\}^{-\zeta}) \bigr) ,
\]
where the sum runs over $l_x , l_y , x'$ and $y'$ 
satisfying $\cd{l_x}(x') \subset \cd{L}(x), \cd{l_y}(y') 
\subset \cd{L}(y) \text{ with } l_x,l_y \in \mathfrak L\cup\{L\}$. 
Recall that the event $A(\cd{l_x}(x'),\cd{l_y}(y'),\epsilon)$ 
is defined in Eq.~\eqref{eq:eventA}. Since $\cd{4l_x}(x')\cap\cd{4l_y}(y')= \emptyset$, Proposition~\ref{prop:unif_wegner_type} provides an upper bound on the probability of $A(\cd{l_x}(x'),\cd{l_y}(y') ,\allowbreak \min\{l_x , l_y\}^{-\zeta})$. This gives
\begin{multline*}
 \PP(\Omega_{\rm W}(x,y)) \\ \leq \sum\limits_{l_x , l_y , x' , y'} C_1 (2 \max \{ l_x,l_y \} + 1)^{3d + \lVert I_0 \rVert_1} \bigl[ \min\{l_x , l_y\}^{-\zeta} + C_2 \euler^{-3 \min \{l_x,l_y\} \alpha / 2} \bigr] .
\end{multline*}
There is an $l_2^* = l_2^* (\kappa)$ such that $24l+2 \leq L$ for $l \geq l_2^*$. 
Recall that $l = L^{1/\kappa}$. 
For $l \geq l_2^*$ and $l_x,l_y \in \mathfrak L\cup\{L\}$ 
we have $8l \leq l_x,l_y \leq L$ 
as well as
\begin{align}
 \PP(\Omega_{\rm W}(x,y)) &\leq 16 (2L+1)^{2d} C_1 (2 L + 1)^{3d+\lVert I_0 \rVert_1} \bigl[ l^{-\zeta} + C_2 \euler^{-12l \alpha} \bigr] \nonumber \\[1ex]
 &\leq 16 C_1 (2 L + 1)^{5d+\lVert I_0 \rVert_1} \bigl[ L^{-\zeta / \kappa} + C_2 \euler^{-12l \alpha} \bigr] \nonumber \\[1ex]
 &\leq 16 C_1 (2 L + 1)^{5d+\lVert I_0 \rVert_1 - \zeta / \kappa} 
 + 16 C_1 C_2 (2 L + 1)^{5d+\lVert I_0 \rVert_1} \euler^{-12l \alpha}  . \label{eq:l3*}
\end{align}
Since $\zeta = \kappa (5d + \lVert I_0 \rVert_1 + 2 \xi)+1$ we find $l_3^* = l_3^* (\xi , \kappa , d , u, \rho)$, 
such that for $l \geq l_3^*$ we have
\[
 \PP(\Omega_{\rm W}(u_1,u_2))\leq \dfrac{1}{3}L^{-2\xi}.
\]
For all $x, y \in \ZZ^d$ with $\lVert x-y\rVert_{\infty} > 8L$ 
we get due to subadditivity of the measure
\[
 \PP(\Omega^{\rm c}_G(x) \cup \Omega^{\rm c}_G(y) \cup \Omega_{\rm W}(x,y)) \leq L^{-2\xi}
\]
and thus 
\[
  \PP(\Omega_{\rm G}(x) \cap \Omega_{\rm G}(y) \cap \Omega_{\rm W}^{\rm c}(x,y)) \geq 1-L^{-2\xi} .
\]
The probability of this event becomes bigger with growing $L$. 
\par
If $\omega \in \Omega_{\rm W}^{\rm c}(x,y)$ and $E \in \RR$, we show that
for one of the cubes $\cd{L}(x)$ or $\cd{L}(y)$, denoted by
$\cd{L}(z)$, all contained cubes $\cd{l_z}(z') \subset \cd{L} (z)$
with $l_z \in \mathfrak L \cup \{L \}$ are uniformly $E$-NR.
If all such cubes both in $\cd{L} (x)$ and $\cd{L} (y)$ are uniformly $E$-NR there is nothing to prove. Otherwise choose $\tilde l \in \mathfrak{L} \cup \{L\}$ maximal such that there exists $\cd{\tilde l} (\tilde z)$ in $\cd{L} (x)$ or $\cd{L} (y)$ with
\[
 \tilde d \Bigl( E , \sigma \bigl( h_{\omega}^{\cd{\tilde l} (\tilde z)} \bigr) \Bigr) < \frac{1}{2} \tilde l^{-\zeta} .
\]
Without loss of generality let $\tilde z \in \cd{L} (y)$. 
By definition all cubes $\cd{l_x} (x') \subset \cd{L} (x)$ with $l_x > \tilde l$ are uniformly $E$-NR. 
Now choose $\cd{l_x} (x') \subset \cd{L} (x)$ with $l_x \leq \tilde l$. 
Since $\omega \in \Omega_{\rm W}^{\rm c}(x,y)$
\begin{align*}
 \tilde d \Bigl(E , \sigma \bigl(h_\omega^{\cd{l_x}(x')}\bigr) \Bigr) 
&\geq \tilde d \Bigl(\sigma \bigl(h_\omega^{\cd{\tilde l}(\tilde z)}\bigr) , \sigma \bigl(h_\omega^{\cd{l_x}(x')}\bigr)\Bigr) 
-  \tilde d \Bigl( E , \sigma \bigl(h_\omega^{\cd{\tilde l}(\tilde z)} \bigr)\Bigr)  \\[1ex]
& \geq \min\{l_x , \tilde l\}^{-\zeta} - \frac{1}{2} \tilde l^{-\zeta} \geq \frac{1}{2} l_x^{-\zeta} ,
\end{align*}
i.e.\ $\cd{l_x} (x')$ is uniformly $E$-NR. 
\par
For $\Vert x-y \Vert_\infty > 8L$ and $\omega \in \Omega_{\rm G}(x) \cap
\Omega_{\rm G}(y) \cap \Omega_{\rm W}^{\rm c}(x,y)$, the two $L$-cubes $\cd{L}(x)$
and $\cd{L}(y)$ have both no four bad $l$-cubes
with distance among them bigger than $8l$ and for at least one of the
$L$-cubes all bad cubes are uniformly $\mathrm{E-NR}$.
\par
From now on we fix $x,y \in \ZZ^d$ with $\lVert x-y \rVert_\infty > 8L$
and $\omega \in \Omega_{\rm G}(x) \cap \Omega_{\rm G}(y) \cap \Omega_{\rm W}^{\rm
  c}(x,y)$. Without loss of generality we assume that the $L$-cube
$\cd{L}(z),\ z \in \{x,y\}$ with the uniformly $\mathrm{E-NR}$ cubes
is equal to $\cd{L}(x)$.  In the next step we show 
that $\cd{L}(x)$ is itself a uniformly $(m_L,E)$-regular cube, with
$m_L$ satisfying \eqref{eq:m_L} in the theorem.  To do so, we need a
sufficiently good estimate on $|G_\omega^{\cd{L}(x)}(E;x,v)|$ for all
$v \in \partial^{i}\cd{L}(x)$.
\par
Let us first recall the geometric resolvent identity. For all $u \in
\mathfrak{C} \subset \cd{L}(x), \ v \in \cd{L}(x)\setminus
\mathfrak{C}$, the following identity holds: 
\begin{equation*}
 G_\omega^{\cd{L}(x)}(E;u,v) = \sum\limits_{\substack{(w,w') \in \partial^{\rm b}\mathfrak{C}}}G_\omega^{\mathfrak{C}}(E;u,w)G_\omega^{\cd{L}(x)}(E;w',v).
\end{equation*}
Here we use the convention that $G_\omega^{\Gamma} (E;x,y) = 0$ if $x \not \in \Gamma \subset \ZZ^d$ or $y \not \in \Gamma$.
With $w_0 \in (\partial^{o}\cd{})\cap\cd{L}(x)$ such that
\[|G_\omega^{\cd{L}(x)}(E;w_0,v)|=\sup\limits_{w' \in (\partial^{\rm o}\mathfrak{C})\cap\cd{L}(x)}|G_\omega^{\cd{L}(x)}(E;w',v)|,\] 
we obtain
\begin{equation}\label{eq:res_id_general}
 \lvert G_\omega^{\cd{L}(x)}(E;u,v) \rvert \leq \left[\sum\limits_{(w,w') \in \partial^{\rm b}\mathfrak{C}} \lvert G_\omega^{\mathfrak{C}}(E;u,w) \rvert \right] \lvert G_\omega^{\cd{L}(x)}(E;w_0,v) \rvert .
\end{equation}
\par
Now we set $u_0\colonequals x$ and fix $v \in \partial^{\rm i}\cd{L}(x)$. Using \eqref{eq:res_id_general} we recursively introduce a sequence $u_1, u_2,\dots \in \cd{L}(x)$ by distinguishing two cases at each step. 
Given $u_k \in \cd{L}(x)$, we construct $u_{k+1}$ according to the following two cases:
\begin{enumerate}[(a)]
\item 
\label{case_a_res_id}
$\cd{l}(u_k)$ is a uniformly $(m,E)$-regular cube and $v \not \in \cd{l} (u_k)$. Using~\eqref{eq:res_id_general} with $\cd{}~= ~\cd{l}(u_k) \cap \cd{L} (x)$, and setting $u_{k+1}\colonequals w_0$, we obtain
\[
 |G_\omega^{\cd{L}(x)}(E;u_k,v)| \leq  2^d d (l+1)^{d-1}\euler^{-ml} |G_\omega^{\cd{L}(x)}(E;u_{k+1},v)|.
\]
Note that $\lVert u_{k+1}- u_k\rVert_{\infty} \leq l+1$. In this case we set 
\[Z(k)~\colonequals~\exp[-ml + \ln(2^dd(l+1)^{d-1})].\]
There is an $l_4^*=l_4^*(d,\beta)$ such that $Z(k)<1$ for all $l\geq l_4^*$.
\item 
\label{case_b_res_id}
$\cd{l}(u_k)$ is not uniformly $(m,E)$-regular. This means that $u_k \in \mathfrak{Ct}_i$ for some $i$, and we assume that $v \not \in \mathfrak{Ct}_i$. Then, using \eqref{eq:res_id_general} with $\cd{} = \mathfrak{Ct}_i \cap \cd{L} (x)$, we obtain
\[
 \lvert G_\omega^{\cd{L}(x)}(E;u_k,v) \rvert \leq  2^{d+1} d (l_i+1)^{d-1}l_i^{\zeta} \lvert G_\omega^{\cd{L}(x)}(E;w_0,v) \rvert,
\]
with $\cd{l}(w_0)$ a uniformly $(m,E)$-regular cube. If we assume further that $v \not \in \cd{l} (w_0)$, we can apply \eqref{eq:res_id_general} again with $\cd{} = \cd{l}(w_0) \cap \cd{L} (x)$, and obtain finally 
\[
  |G_\omega^{\cd{L}(x)}(E;u_k,v)| \leq Z(k)|G_\omega^{\cd{L}(x)}(E;u_{k+1},v)|,
\]
with 
\[
 Z(k)\colonequals 2^{2d+1}d^2[(l_i+1)(l+1)]^{d-1}l_i^{\zeta}\euler^{-ml} .
\]
Note that here $\lVert u_{k+1} - u_{k}\rVert_{\infty} \leq 2l_i + l+3$. A straightforward calculation shows that $Z(k)~<~1$ is satisfied if
\begin{equation}
m > \frac{\ln C(\zeta,d)}{l} + (2d-2+\zeta)\frac{\ln l}{l},
\label{eq:m_lower_bound}
\end{equation}
with a constant $C(\zeta,d) > 0$ only depending on $\zeta$ and $d$.
The assumption $m>l^{\beta-1}$ of the theorem guarantees \eqref{eq:m_lower_bound} for $l\geq l_5^*$ with a suitably chosen $l_5^*=l_5^*(d,\beta,\zeta)$.
\end{enumerate}
If none of the above two cases applies for a given $u_k$, we cannot construct $u_{k+1}$. We assume now that it is possible to perform $n$ steps of the recursion with associated sites $u_0,u_1,\dots, u_n \in \cd{L}(x)$. Applying the
non-resonance of $\cd{L}(x)$, we obtain
\begin{equation}
  |G_\omega^{\cd{L}(x)}(E;x,v)| \leq \left(\prod_{k=0}^{n-1}Z(k)\right)|G_\omega^{\cd{L}(x)}(E;u_n,v)|\leq \left(\prod_{k=0}^{n-1}Z(k)\right)2L^{\zeta}.
\label{eq:greens_function_ub_enr}
\end{equation}
If it is possible to perform arbitrarily many steps of the iteration
without leaving $\cd{L}(x)$, it follows from
\eqref{eq:greens_function_ub_enr} for $L\geq l_4^*,l_5^*$ that
$|G_\omega^{\cd{L}(x)}(E;x,v)| =0$. Otherwise, the iteration
terminates after finitely many steps, i.e., for some $k \in \NN$ the
site $u_k \in \cd{L}(x)$ is so close to the boundary of $\cd{L}(x)$
such that the assumption of neither (a) nor (b) are satisfied.

In this latter case, we can give a lower bound on the number $n_1$
of case (\ref{case_a_res_id}) steps
performed before the recursion ends.
Using the estimates for $\lVert u_{k+1}-u_{k}\rVert_{\infty}$ and Eq.~\eqref{eq:cont_length}, we obtain
\[
 n_1 \geq \frac{L-l-\sum_{i=1}^N(2l_i+l+3)}{l+1} \geq \frac{L-63l}{l+1}.
\]
Using \eqref{eq:greens_function_ub_enr} and disregarding the case (\ref{case_b_res_id}) steps we have
\begin{align*}
 |G_\omega^{\cd{L}(x)}(E;x,v)| &\leq \exp\left[\left(\frac{L-63l}{l+1}\right)\left(-ml + \ln(2^dd(l+1)^{d-1})\right) +\ln( 2L^\zeta)\right]\\
&\leq \exp\left[ -mL + m\frac{L}{l} + 63ml + \frac{L}{l}\ln(2^d d (l+1)^d) +  \ln( 2L^\zeta)\right]\\
&\leq \exp\left[ -m_L L\right],
\end{align*}
with 
\[
 m_L\colonequals m \Bigl(1-2^{\frac{1}{\kappa}} L^{-\frac{1}{\kappa}} -63L^{\frac{1}{\kappa} - 1} \Bigr) - 2^{\frac{1}{\kappa}}L^{-\frac{1}{\kappa}}\ln(4^d d L^{\frac{d}{\kappa}}) - \frac{\ln(2L^{\zeta})}{L},
\]
where we used $l^\kappa = L$.
\par
Now we choose $\gamma \in ((1-\beta) / \kappa ,1-1 /  \kappa)$. This is possible because of $2-\kappa < \beta$. Since $1-1 / \kappa  < 1 / \kappa$, 
we have
  \[
  m_L \geq m(1-L^{-\gamma}) - L^{-\gamma},
  \]
  for all $L^{\frac{1}{\kappa}}\geq l \geq l_6^*$ with appropriate
  $l_6^*=l_6^*(d,\kappa,\zeta,\gamma)$.  Using $m\geq l^{\beta-1}$ we
  conclude
  \[
  m_L \geq
  L^{-\frac{1-\beta}{\kappa}}-L^{-\gamma-\frac{1-\beta}{\kappa}}-L^{-\gamma}.
  \]

Since $\beta > 1-\gamma \kappa$ we can find $l_7^*=l_7^*(\beta,\gamma,\kappa)$ such that for $L^{1 / \kappa} \geq l \geq l_7^*$ we have
\[
 m_L > L^{\beta -1}
\]
The theorem follows with $l^*\colonequals\max({l_1^*,\dots,l_7^*})$.
\end{proof}
\subsection{Localization; proof of Theorem~\ref{thm:loc_under_ini}}
In Section~\ref{sec:induction} we carried out the induction step of the multiscale analysis, i.e.\ that if $G(I , l_1 , m_1 , \xi)$ holds for some $l_1 >0$, then $G(I , l_2 , m_2 , \xi)$ holds on some larger scale $l_2 > l_1$. Once an induction anchor is given, one obtains the estimate $G(I , l_k , m_k , \xi)$ for an increasing sequence of length scales $l_k$. It is crucial for concluding localization that the sequence $m_k$ is bounded from below by some positive $m$. The induction anchor is provided by the so-called initial scale estimate formulated in the following assumption.
\par
Before we define the initial scale estimate let us define a new length scale $\overline l = \overline l (\beta , \kappa, q , m_0) \in \NN$, depending on $\beta, q \in (0,1)$, $\kappa \in (1,2)$ and $m_0 > 0$, namely
\begin{equation} \label{eq:lbar}
 \overline l = \overline l (\beta , \kappa , q , m_0) = \left( \frac{(1-q) m_0}{(1-q)m_0 + m_0 + 1} \right)^{-\kappa / (1-\beta)} .
\end{equation}
\begin{definition} \label{ass:ini}
Let $I \subset \RR$. We say that the \emph{initial scale estimate} holds in $I \subset \RR$, if 
\[
 G (I , l_0 , m_0 , \xi)
\]
is satisfied for some 
\begin{enumerate}[(i)]
 \item $\xi > 2d$, $\kappa \in (1,2\xi/(\xi + 2d))$, $\beta \in (2-\kappa , 1)$, and
 \item $q \in (0,1)$, $m_0 > 0$ and $l_0 >1$ satisfying $l_0 \geq \max\{ l^* , \overline l \}$ and $m_0 > l_0^{\beta - 1}$. 
\end{enumerate}
Here $l^* = l^* (d,\xi,\kappa,\beta,u,\rho)$ is given by Theorem~\ref{thm:induction} and $\overline l = \overline l (\beta , \kappa , q , m_0)$ is as in Eq.~\eqref{eq:lbar}.
\end{definition}
Note that $\overline l$ depends on $m_0$. Hence, if one has verified $G(I,l_0,m_0,\xi)$ for some $l_0 \geq l^*$ and $m_0 > l_0^{\beta-1}$ one still has to check whether $l_0 \geq \overline l$. However, if one has verified $G(I,l_0,m_0,\xi)$ for some $m_0 > 0$ and all $l_0 > 1$, then one just has to choose $l_0$ sufficiently large to verify the initial scale estimate.
\begin{theorem}\label{thm:msaoutput}
Let Assumptions~\ref{ass:exp} and \ref{ass:bv} be satisfied, $I \subset \RR$, and assume that the initial scale estimate holds in $I$. Set $m_\infty = q m_0$. Then
\[
 G (I , l_k , m_\infty , \xi)
\]
holds for all $k \in \NN_0$. Here the sequence $l_k$, $k \in \NN_0$, is defined by 
\[
 l_{k + 1} = l_k^\kappa , \quad k \in \NN_0,
\]
and $l_0$, $m_0$, $q$, $\kappa$ and $\xi$ are given through the initial scale estimate.
\end{theorem}
\begin{proof}
By assumption of our theorem, the hypothesis of Theorem~\ref{thm:induction} is satisfied with $l = l_0$ and $m = m_0$. By an inductive application of Theorem~\ref{thm:induction} we obtain the estimate $G(I,l_k,m_k,\xi)$ with a decreasing sequence $m_k$, $k \in \NN_0$, satisfying
\[
 m_{k+1} \geq m_k \bigl(1-l_{k+1}^{-(1-\beta)/\kappa} \bigr) - l_{k+1}^{-(1-\beta)/\kappa} .
\]
Our theorem now follows if
\begin{equation} \label{eq:masseverlust}
 \sum_{k=0}^\infty m_k - m_{k+1} \leq m_0 - m_\infty = (1-q) m_0 .
\end{equation}
Indeed, we can estimate
\begin{align*}
  \sum_{k=0}^\infty m_k - m_{k+1} & = \sum_{k=0}^\infty m_k l_{k+1}^{-(1-\beta)/\kappa} + \sum_{k=0}^\infty l_{k+1}^{-(1-\beta)/\kappa} \leq (m_0 + 1) \sum_{k=0}^\infty  l_{k+1}^{-(1-\beta)/\kappa} .
\end{align*}
Since this is a geometric series we obtain
\[
  \sum_{k=0}^\infty m_k - m_{k+1} \leq (m_0 + 1) \frac{l_0^{-(1-\beta)/\kappa}}{1 - l_0^{-(1-\beta)/\kappa}} .
\]
By assumption we have $l_0 \geq \overline l$, and by definition of $\overline l$ we obtain Ineq.~\eqref{eq:masseverlust} and hence the statement of the theorem.
\end{proof}
Next we cite \cite[Theorem~2.3]{DreifusK-89}. More precisely, we will state a slight generalization, since \cite[Theorem~2.3]{DreifusK-89} was stated for the case $u = \delta_0$ only. In particular, the proof applies directly to the case of general single-site potentials $u$ and the measure $\mu$. Even more, this result holds true for arbitrary potentials, as long as the resulting family is a family of self-adjoint operators. To be more precise, consider the family of self-adjoint operators 
\[
A_\omega : \ell^2 (\ZZ^d) \to \ell^2 (\ZZ^d), \quad \omega \in \Omega ,
\]
where the index $\omega$ is an element of some probability space $(\tilde \Omega , \tilde{\mathcal{F}} , \tilde \PP)$. We assume that the map $\omega \mapsto \langle \phi , (A_\omega - z)^{-1} \psi \rangle$ is measurable for all $\psi,\phi \in \ell^2 (\ZZ^d)$ and all $z \in \CC \setminus \RR$. As supplied before we use similar notation for the restricted operators $A_\omega^{\cd{}} : \ell^2 (\cd{}) \to \ell^2 (\cd{})$ and, with some abuse of notation, the symbol $G_\omega^{\cd{}} (E,u,w) = \langle \delta_u (A_\omega^{\cd{}} - E)^{-1} \delta_w \rangle$ for the corresponding Green function. Moreover, the definition of $(m,E)$-regular and singular from Definition~\ref{def:mereg} holds for $A_\omega^{\cd{}}$ in an analogue way.
\begin{theorem} \label{thm:vDK-2.3}
Consider the family of operators $(A_\omega)_{\omega \in \tilde \Omega}$, let $a\in \NN$, $c \in \NN_0$, $I\subset \RR$ be an interval, $\xi>d$, $ l_0>1$, $\kappa \in (1,2\xi/d)$ and $m>0$. Let moreover $l_k$, $k \in \NN_0$, be a sequence of integers such that for all $k \in \NN_0$
\[
 l_{k+1} = l_k^\kappa .
\]
Suppose that for any $k \in \NN_0$ and any $x,y \in \ZZ^d$ with $\lVert x-y\rVert_\infty \geq al_k + c$
\begin{equation*}
 \tilde\PP \bigl( \{\omega \in \tilde\Omega \colon \exists \, E \in I \colon \text{$\cd{l_k}(x)$ and $\cd{l_k}(y)$ is $(m,E)$-singular} \} \bigr) \leq l_k^{-2\xi} .
\end{equation*}
Then, for almost all $\omega\in\tilde\Omega$, $\sigma_{c} (A_\omega) \cap I = \emptyset$ and the eigenfunctions corresponding to the eigenvalues of $A_\omega$ in $I$ decay exponentially.
\end{theorem}
\begin{proof}
 Let $b$ be a positive integer to be chosen later on. For $x_0 \in \ZZ^d$ and $k \in \NN_0$ let
\[
 A_{k+1} (x_0) = \cd{b(a l_{k+1} + c)}(x_0) \setminus \cd{al_{k} + c}(x_0) .
\]
Define the event
\begin{multline*}
 E_k (x_0 ) = \{\omega \in \tilde\Omega \colon \text{$\cd{l_k}(x_0)$ and $\cd{l_k}(x)$ are $(m,E)$-singular} \\ \text{for some $E \in I$ and some $x \in A_{k+1}(x_0)$}\} .
\end{multline*}
By construction we have for each $x \in A_{k+1}(x_0)$ that $\lVert x-x_0 \rVert_\infty > a l_k + c$. Hence, we obtain by our hypothesis
\begin{align*}
 \tilde\PP \bigl( E_k (x_0) \bigr) &\leq \!\!\!\! \sum_{x \in A_{k+1}(x_0)} \!\!\!\! l_k^{-2\xi} 
 &\leq \frac{(2(bal_{k+1} + bc) + 1)^d}{l_k^{2\xi}} \leq \frac{(2ba + 2bc l_0^{-1} + l_0^{-1})^d}{l_k^{2\xi - \kappa d}}
\end{align*}
for all $k \in \NN_0$. Since $\kappa d < 2 \xi $ we have $\sum_{k=0}^\infty \tilde\PP ( E_k (x_0) )  < \infty$. It follows from Borel Cantelli Lemma that for each $x_0 \in \ZZ^d$ we have $\tilde\PP \{ E_k (x_0)$ occurs infinitely often$\}=0$. Since a countable union of sets of measure zero has measure zero, we obtain
\[
 \tilde\PP \bigl(\{\omega \in \tilde\Omega \colon \exists \, x_0 \in \ZZ^d \colon \text{$E_k (x_0)$ occurs for infinitely many $k \in \NN$} \}\bigr) = 0 .
\]
If we let 
\[
\tilde\Omega_0 = \{\omega \in \tilde\Omega \colon \text{for all $x_0 \in \ZZ^d$, $E_k (x_0)$ occurs only finitely many times} \} , 
\]
we have $\tilde\PP (\tilde\Omega_0) = 1$. In particular, for each $\omega \in \tilde\Omega_0$ and $x_0 \in \ZZ^d$ there is $k_1 = k_1 (\omega, x_0) \in \NN$ such that if $k \geq k_1$ then $E_k (x_0)$ does not occur.
\par
Now let $\omega \in \tilde\Omega_0$, $E \in I$ be a generalized eigenvalue of $A_\omega$ with the corresponding non-zero polynomially bounded generalized eigenfunction $\psi$, i.e.\ $A_\omega \psi = E \psi$, $\lvert \psi (x) \rvert \leq C (1 + \lVert x \rVert)^t$ for some positive constant $C$ and positive integer $t$. We choose $x_0 \in \ZZ^d$ such that $\psi (x_0) \not = 0$.
If $\cd{l_k} (x_0)$ is $(m,E)$-regular, then $$E \not \in \sigma \Bigl(A_\omega^{\cd{l_k} (x_0)}\Bigr)$$ and therefore we can recover $\psi$ from its boundary values, i.e.
\begin{align} \label{eq:boundary}
 \lvert \psi (x_0) \rvert &= \Biggl\lvert \sum_{i \in \partial^{\rm i} \cd{l_k}(x_0)} G_{\cd{l_k}(x_0)} (E;x_0,i) \sum_{\genfrac{}{}{0pt}{}{y \in (\cd{l_k}(x_0))^{\rm c} :}{\lVert i-y \rVert_1 = 1}} \psi (y) \Biggr\rvert\\
&\leq \sum_{i \in \partial^{\rm i} \cd{l_k}(x_0)} \euler^{-ml_k} 2d C(2+l_k+\lVert x_0 \rVert)^t \nonumber.
\end{align}
Since $\psi (x_0) \not = 0$, it follows that there exists $k_2 = k_2 (\omega , E, x_0) \in \NN$ such that $\cd{l_k}(x_0)$ is $(m,E)$-singular for all $k \geq k_2$. Let $k_3 = k_3 (\omega,E,x_0) = \max\{k_1,k_2\}$. If $k \geq k_3$ we conclude that $\cd{l_k}(x)$ is $(m,E)$-regular for all $x \in A_{k+1} (x_0)$.
\par
Let $\rho \in (0,1)$ be given. We pick $b > (1 + \rho) /(1-\rho)$ and define
\[
 A_{k+1}' (x_0) = \cd{b(a l_{k+1} + c)/(1+\rho)}(x_0) \setminus \cd{(al_{k} + c)/(1-\rho)}(x_0) .
\]
Then we have
\begin{enumerate}[(i)]
 \item $A_{k+1}' (x_0) \subset A_{k+1} (x_0)$ for $k \in \NN_0$,
 \item if $x \in A_{k+1}' (x_0)$ then $\dist (x,\partial^{\rm o} A_{k+1} (x_0)) \geq \rho \lVert x-x_0 \rVert_\infty$, and
 \item if $x \not \in \cd{(al_0 + c)/(\rho - 1)}(x_0)$ then $x \in A_{k+1}'$ for some $k \in \NN_0$.
\end{enumerate}
Here $\dist(m,A) = \inf_{k \in A} \lVert m-k \rVert_{\infty}$ for $k \in \ZZ^d$ and $A \subset \ZZ^d$.
Claim (i) and (iii) are obvious. To see (ii) we estimate the distance of $x\in A_{k+1}' (x_0)$ to both boundaries of the annulus $A_{k+1} (x_0)$. For the ``inner'' boundary we use $\lVert x-x_0 \rVert_\infty \leq \lfloor al_k + c \rfloor + \dist (x,\partial^{\rm i}\cd{al_k + c}(x_0))$ and $\lVert x-x_0 \rVert_\infty \geq \lfloor al_k + c \rfloor/(1-\rho)$ to conclude 
\[
\dist (x ,\partial^{\rm i} \cd{al_k + c}(x_0)) \geq \lVert x-x_0 \rVert_\infty - (1-\rho) \lVert x-x_0 \rVert_\infty .
\]
For the ``outer'' boundary we use the triangle inequality $$\dist (x_0 , \partial^{\rm o} \cd{b(al_{k+1}+c)}(x_0)) \leq \lVert x-x_0 \rVert_\infty + \dist (x,\partial^{\rm o} \cd{b(al_{k+1}+c)}(x_0)),$$ $\lVert x-x_0 \rVert_\infty \leq b(al_{k+1}+c)/(1+\rho)$ and $\dist(x_0, \partial^{\rm o} \cd{b(al_{k+1}+c)}(x_0)) = b(al_{k+1}+c)$ to conclude
\begin{align*}
 \dist (x,\partial^{\rm o} \cd{b(al_{k+1}+c)}(x_0)) &\geq \dist (x_0 , \partial^{\rm o} \cd{b(al_{k+1}+c)}(x_0)) -  \lVert x-x_0 \rVert_\infty \\
 &\geq \rho \lVert x-x_0 \rVert_\infty .
\end{align*}
Hence the claim (ii) follows.
\par
Now let $k \geq k_3$, so that $\cd{l_k}(y)$ is $(m,E)$-regular for any $y \in A_{k+1} (x_0)$. Let $x \in A_{k+1}' (x_0) \subset A_{k+1} (x_0)$. Again by Eq.~\eqref{eq:boundary},
\[
 \lvert \psi (x) \rvert \leq (2l_k + 1)^d \euler^{-ml_k} 2d \lvert \psi (u_1) \rvert
\]
for some $u_1 \in \partial^{\rm o} \cd{l_k + 1}(x)$. If $u_1 \in A_{k+1} (x_0)$ we obtain
\[
  \lvert \psi (x) \rvert \leq \bigl[(2l_k + 1)^d \euler^{-ml_k} 2d\bigr]^2 \lvert \psi (u_2) \rvert
\]
for some $u_2 \in \partial^{\rm o} \cd{l_k}(u_1)$. By claim (ii) we can repeat this procedure at least $\lfloor\rho \lVert x-x_0 \rVert_\infty /(l_k +1)\rfloor$ times, use the polynomial bound on $\psi$ and obtain for all $k \geq k_3$ and all $x \in A_{k+1}' (x_0)$ the inequality
\[
  \lvert \psi (x) \rvert \leq \bigl[(2l_k + 1)^d \euler^{-ml_k} 2d\bigr]^{\left\lfloor\rho \lVert x-x_0 \rVert_\infty /( l_k  +1)\right\rfloor} C \bigl(1+\lVert x_0 \rVert_\infty + b(al_{k+1} + c) \bigr)^t .
\]
We can rewrite the above inequality  as
\begin{multline*}
  \lvert \psi (x) \rvert \leq 
\exp\Biggl\{ -  \left\lfloor \frac{\rho \lVert x-x_0 \rVert_\infty}{ l_k + 1} \right\rfloor \rho m l_k \Biggl\} 
\exp\Biggl\{  \left\lfloor \frac{\rho \lVert x-x_0 \rVert_\infty}{l_k  + 1} \right\rfloor \bigg[d \ln (2l_k + 1)\Biggr. \\ + \ln (2d)  - (1-\rho)m l_k \bigg]   \Biggl.+ t \ln \left(C(1+\lVert x_0 \rVert_\infty + b(al_{k+1} + c) )\right) \Biggr\} .
\end{multline*}
Since $(al_k + c)/(1-\rho)\leq\lVert x-x_0 \rVert_\infty \leq b(a l_k^\kappa + c)/(1+\rho)$, the second exponential function gets smaller than one if $k$ is larger than some suitable $k_4$. Let $\rho' \in (0,1)$ and choose $\rho$ such that $\rho > 1/(1+a-\rho' a)$. We obtain that the first exponential function is bounded from above by
\begin{align*}
&\phantom{\leq} \exp\Biggl\{ -  \left( \frac{\rho \lVert x-x_0 \rVert_\infty}{l_k + 1} - 1 \right ) \rho m l_k \Biggl\}\\ 
&\leq \exp \Biggl \{ \rho m l_k \Biggr\} \exp \Biggl \{- \rho^2 m\lVert x-x_0 \rVert_\infty \frac{l_k}{l_k + 1} \Biggr\} \\
& \leq \exp \Biggl \{ \rho m l_k \biggl[ 1 - (1-\rho') \rho  \frac{\lVert x-x_0 \rVert_\infty}{l_k + 1} \biggr] \Biggr\} \exp \Biggl \{- \rho^2 \rho' m\lVert x-x_0 \rVert_\infty \frac{l_k}{l_k + 1} \Biggr\} .
\end{align*}
Again, using the lower bound on $\lVert x-x_0 \rVert_\infty$ and the relation between $\rho$ and $\rho'$, we see that the first exponential function gets smaller than one if $k \geq k_5$ with appropriate $k_5$. Hence, if we pick $\rho'' \in (0,1)$ we find $k_6 \in \NN$ such that for all $k \geq k_6$ and all $x \in A_{k+1}' (x_0)$ we have
\begin{equation} \label{eq:ende}
 \lvert \psi (x) \rvert \leq \exp \Bigl \{- \rho^2 \rho' m\lVert x-x_0 \rVert_\infty \rho'' \Bigr\} .
\end{equation}
Set $k_7 = \max\{k_1 , \ldots , k_6\}$. By claim (iii) we conclude that for all $x \in \ZZ^d \setminus \cd{(al_{k_7} + c)/(1-\rho)}(x_0)$ we have Ineq.~\eqref{eq:ende}.
\par
We have shown for all $\omega \in \tilde\Omega_0$ that every generalized
eigenvalue is an eigenvalue with an exponentially decaying eigenfunction. To end the proof we use the fact that for any $\omega \in \tilde\Omega$, 
almost every energy $E$ (with respect to a spectral measure) is a generalized eigenvalue \cite{Berezanskii-68,Simon-82}, see also \cite[Proposition~7.4]{Kirsch-08}. 
\par
Fix $\omega \in \tilde\Omega_0$, let $M_0 \subset I$ be the set of all generalized eigenvalues in $I$ and $M_1$ be the set of all eigenvalues in $I$. It follows that $I \setminus M_0$ has $\rho_\omega$-measure zero, and since $M_0 \subset M_1$ we conclude that $I \setminus M_1$ has $\rho_\omega$-measure zero. 
Since $\ell^2 (\ZZ^d)$ is separable $M_1$ is a countable set, and therefore the measure $\rho_\omega$ restricted to $I$ is a pure point measure. Hence, $\sigma_{\rm c} (A_\omega) \cap I = \emptyset$.
\end{proof}
\begin{proof}[Proof of Theorem~\ref{thm:loc_under_ini}]
Note that each cube that is uniformly $(m,E)$-regular for $\omega$ is also $(m,E)$-regular for $\omega$. Hence, as a corollary of Theorem~\ref{thm:msaoutput} and \ref{thm:vDK-2.3} we obtain exponential localization for the discrete alloy-type model $h_\omega$ in any energy region where the initial length scale estimate holds. This proves Theorem~\ref{thm:loc_under_ini}.
\end{proof}
\subsection{Initial scale estimate; proof of Theorem~\ref{thm:loc:large} and \ref{thm:loc:weak}} \label{sec:ini}
In this subsection we prove the initial scale estimate in certain disorder/energy regimes, formulated precisely in Lemma~\ref{lemma:ini_large} and Proposition~\ref{pro:initial_small}. Together with Theorem~\ref{thm:loc_under_ini} we obtain localization as stated in Theorem~\ref{thm:loc:large} and \ref{thm:loc:weak}.
The proofs of Theorem~\ref{thm:loc:large} and \ref{thm:loc:weak} are given at the end of this subsection. 
\par
In the large disorder regime the initial scale estimate can be deduced from the uniform control of resonances, see e.g.\ Theorem~11.1 in \cite{Kirsch-08} or Lemma~14 in \cite{Veselic-10b}. Since we provide uniform control of resonances for our model in Proposition~\ref{prop:unif_wegner_type}, we obtain the following lemma by following \cite{Kirsch-08,Veselic-10b}.
\begin{lemma}[Initial scale estimate, large disorder] \label{lemma:ini_large}
Let Assumptions~\ref{ass:exp} and \ref{ass:bv} be satisfied and $\lVert \rho \rVert_{\rm Var}$ sufficiently small. Then the initial scale estimate is satisfied in $\RR$.
\end{lemma}
\begin{proof} Let $l_0 > 0$, $x,y \in \ZZ^d$ with $\lVert x-y \rVert > 8l_0$, $E \in \RR$ and $m_0 > 0$. Then
 \begin{align*}
  \PP \bigl(B_{l_0} (x,y,m_0,E) \bigr) 
& \leq \PP \left( \exists E \colon \tilde d \Bigl( E , \sigma \Bigl(h_\omega^{\cd{l_0} (z)}\Bigr) \Bigr) < \euler^{m_0 l_0}, \ z \in \{x,y\} \right) \\[1ex]
& \leq \PP \left(\tilde d \Bigl(\sigma \Bigl(h_\omega^{\cd{l_0} (x)}\Bigr) , \sigma \Bigl(h_\omega^{\cd{l_0} (y)} \Bigr) \Bigr) < 2 \euler^{m_0 l_0}  \right) .
 \end{align*}
We use Proposition~\ref{prop:unif_wegner_type} to estimate this probability and obtain
\begin{equation} \label{eq:ini}
 \PP \bigl(B_{l_0} (x,y,m,E) \bigr) \leq C_1 (2 l_0 + 1)^{3d + \lVert I_0 \rVert_1} \bigl[ 2 \euler^{m l_0} + C_2 \euler^{-3 l_0 \alpha / 2} \bigr] .
\end{equation}
From the proof of Proposition~\ref{prop:unif_wegner_type} we infer that $C_1 = \lVert \rho \rVert_{\rm Var} \hat C_1$ with some constant $\hat C_1$ depending only on the single-site potential $u$. We fix $q \in (0,1)$, $\xi > 2d$, $\kappa \in (1, 2\xi / (\xi + 2d))$ and $\beta \in (2 - \kappa , 1)$. Finally, we choose $\lVert \rho \rVert_{\rm Var}$ (small enough) and $l_0>1$ (large enough) in such a way that 
\[
 \PP (B_{l_0} (x,y , m_0,E)) \leq l_0^{-2 \xi}, \quad l_0 \geq \max\{l^* , \overline l \} \quad \text{and} \quad l_0^{\beta - 1} < m_0 .
\]
Here $l^* = l^* (d,\xi,\kappa,\beta,u,\rho)$ is given by Theorem~\ref{thm:induction} and $\overline l = \overline l (\beta , \kappa , q , m_0)$ is as in Eq.~\eqref{eq:lbar}. This choice of $\lVert \rho \rVert_{\rm Var}$ and $l_0$ is always possible as we explain now. Recall that $l^*$ depends on $\lVert \rho \rVert_{\rm Var}$ and note that $l^*$ decreases as $\lVert \rho \rVert_{\rm Var}$ decreases, see Eq.~\eqref{eq:l3*}. Now fix for the moment $\rho$ with $\lVert \rho \rVert_{\rm Var} = 1$ and choose $l_0$ large enough such that $l_0 \geq \max\{l^* , \overline l \}$ and $l_0^{\beta - 1} < m_0$. If we choose $  \rho$ with $\lVert \rho \rVert_{\rm Var} < 1$ these two conditions will still be satisfied since $l^*$ decreases if $\lVert \rho \rVert_{\rm Var}$ decreases. In a last step we choose $\lVert \rho \rVert_{\rm Var} \leq 1$ small enough such that $\PP (B_{l_0} (x,y , m_0,E)) \leq l_0^{-2 \xi}$, which is by Ineq.~\eqref{eq:ini} satisfied if
\[
 \lVert \rho \rVert_{\rm Var}  \leq \frac{l_0^{-2\xi}}{\hat C_1 (2 l_0 + 1)^{3d + \lVert I_0 \rVert_1} \bigl[ 2 \euler^{m l_0} + C_2 \euler^{-3 l_0 \alpha / 2} \bigr]} . \qedhere
\]
\end{proof}
In the case of weak disorder, i.e.\ arbitrary $\lambda > 0$ and energies near the band edges, far less is known if the single-site-potential may change its sign. A version of an initial scale estimate has been proven in \cite{CaoE-12} for exponential decaying sign-changing single-site potentials in the case $d = 3$. However, the initial scale estimate of \cite{CaoE-12} is not suitable to conclude localization via multiscale analysis in the long-range case, since they prove a non-uniform version only. With a non-uniform version the multiscale analysis requires independence at distance, i.e.\ $\supp u$ compact, see e.g.\ \cite[Theorem~2.4]{GerminetK-03}.
\par
In the following we derive a uniform initial scale estimate as formulated in Definition~\ref{ass:ini} for non-compactly supported single-site potentials with a small negative part. The proof is in the manner of \cite{Simon-85} where the Anderson model was considered and \cite{Veselic-01,Veselic-02a} where an initial scale estimate is shown in the case of compactly supported single-site potentials with a small negative part.
\par
For the proof we need to introduce the Neumann Laplacian. The Neumann Laplacian on $\cd{} \subset \ZZ^d$ is the operator on $\ell^2 (\cd{})$ defined by
\[
 h_0^{\cd{} , \rm N} = \pi_{\cd{}}h_0 \iota_{\cd{}} - 2d + n_{\cd{}}
\]
where $n_{\cd{}} : \ell^2 (\cd{}) \to \ell^2 (\cd{})$ is diagonal with
\[
 n_{\cd{}} (i) = \lvert \{ j \in \cd{} \colon \lVert j-i \rVert_1 = 1 \} \rvert .
\]
We denote by $h_\omega^{\cd{} , \rm N} = h_0^{\cd{} , \rm N} + \pi_{\cd{}} v_{\omega} \iota_{\cd{}}$ the corresponding Neumann Hamiltonian.
\begin{proposition} \label{prop:first} Let Assumption~\ref{ass:exp} be satisfied and $\overline{u}\colonequals \sum_{k \in \ZZ^d} u(k) > 0$. There exists $1<\beta_0$ and $l_8^* < \infty$ depending only on $u$, such that if we pick $l \geq l_8^*$ and $\beta \geq \beta_0$, and assume that Assumption~\ref{ass:small_neg} is satisfied for $\delta = l^{-2} / (8 \beta \omega_+)$, then we have the implication
\[
 \lambda_1 (h_\omega^{\cd{l} , \rm N}) < l^{-2} / \beta \quad \Rightarrow \quad
 \biggl\lvert \biggl\{k \in \cd{l} \colon \omega_k < \frac{4l^{-2}}{\beta \overline{u}} \biggr\} \biggr\rvert > \frac{13}{12} \frac{\lvert \cd{l} \rvert}{2} .
\]
\end{proposition}
\begin{proof}
Recall that the eigenvector corresponding to the lowest eigenvalue $\lambda_1 (h_0^{\cd{l} , \rm N}) = 0$ is $\psi \in \ell^2 (\cd{l})$ with $\psi (k) = 1 / \sqrt{\lvert \cd{l} \rvert}$ for all $k \in \cd{l}$. Moreover, the second eigenvalue satisfies the estimate
\begin{equation} \label{eq:gap}
 \lambda_2 (h_0^{\cd{l} , \rm N}) - \lambda_1 (h_0^{\cd{l} , \rm N}) = \lambda_2 (h_0^{\cd{l} , \rm N}) = 2-2\cos(\pi / l) > 4 l^{-2} ,
\end{equation}
see e.g.\ \cite{Simon-85}.
 For $\omega \in \Omega$ we set $\tilde \omega_k =                 \min (\omega_k , 8l^{-2} / (\beta \lVert u \rVert_1))$.
Then we have for all $\omega \in \Omega$ and $x \in \cd{l}$
\begin{align*}
 v_{\tilde\omega} (x) - v_\omega (x) &= \sum_{k \in \ZZ^d} \bigl(\omega_k - \tilde \omega_k \bigr) \bigl( \delta u_- (x-k) - u_+ (x-k) \bigr) 
 \leq \delta \omega_+ = \frac{l^{-2}}{8\beta} .
\end{align*}
This gives $\lambda_1 (h_\omega^{\cd{l},\rm N}) \geq \lambda_1 (h_{\tilde\omega}^{\cd{l},\rm N}) - l^{-2} /(8\beta)$. We want to apply Temple's inequality, 
see e.g.\ \cite{ReedS-78d}, to the operator $h_{\tilde\omega}^{\cd{l},\rm N}$ 
with the vector $\psi$ and some constant $\xi$ with $\langle \psi , h_{\tilde\omega}^{\cd{l},\rm N} \psi \rangle < \xi \leq \lambda_2 (h_{\tilde\omega}^{\cd{l},\rm N})$, in order to estimate $\lambda_1 (h_{\tilde\omega}^{\cd{l},\rm N})$ from below. We set 
\[
 \xi \colonequals \lambda_2 (h_0^{\cd{l} , \rm N}) - \frac{l^{-2}}{8 \beta} \quad \text{and} \quad \beta_0 \colonequals 65 / 32 + 8 \lVert u \rVert_1 / \overline{u} .
\]
That $\xi$ is larger than $\langle \psi , h_{\tilde\omega}^{\cd{l},\rm N} \psi \rangle$ follows from Ineq.~\eqref{eq:gap}, $\beta > 65 / 32$, and the upper bound
\begin{equation} \label{eq:upper1}
 \bigl\langle \psi , h_{\tilde \omega}^{\cd{l} , \rm N} \psi \bigr\rangle = \bigl\langle \psi , \pi_{\cd{l}} v_{\tilde\omega} \iota_{\cd{l}} \psi \bigl\rangle
\leq \frac{1}{\lvert \cd{l} \rvert} \sum_{k \in \cd{l}} \sum_{i \in \ZZ^d} \tilde\omega_i \lvert u(k-i) \rvert  \leq \frac{8 l^{-2}}{\beta} .
\end{equation}
That $\xi$ is smaller or equal to $\lambda_2 (h_{\tilde\omega}^{\cd{l} , \rm N})$ follows from the lower bound
\[
 v_{\tilde \omega} (x) \geq - \sum_{k \in \ZZ^d} \tilde\omega_k \delta u_- (x-k) \geq - \frac{l^{-2}}{8 \beta} \quad \Rightarrow \quad \lambda_2 (h_{\tilde\omega}^{\cd{l} , \rm N}) \geq \lambda_2 (h_0^{\cd{l} , \rm N}) - \frac{l^{-2}}{8\beta} .
\]
Here we have used that $\delta = l^{-2} / (8 \beta \omega_+)$, $\tilde \omega_k \leq \omega_+$ and $\lVert u_- \rVert_1 \leq 1$.
By Ineq.~\eqref{eq:gap}, Ineq.~\eqref{eq:upper1} and the choice of $\beta_0$ we have
\begin{equation} \label{eq:lower2}
\xi -  \bigl\langle \psi , h_{\tilde \omega}^{\cd{l} , \rm N} \psi \bigr\rangle > \frac{l^{-2}}{\beta} \Bigl[ 4 \beta - \frac{1}{8} - 8 \Bigr] \geq \frac{l^{-2}}{\beta} \frac{32 \lVert u \rVert_1}{\overline{u}} .
\end{equation}
For the expectation of the square of $h_{\tilde \omega}^{\cd{l} , \rm N}$ we calculate that
\begin{align*}
 \Bigl\langle \psi , \bigl(h_{\tilde \omega}^{\cd{l} , \rm N} \bigr)^2 \psi \Bigr\rangle \leq \frac{1}{\lvert \cd{l} \rvert} \frac{8l^{-2}}{\beta} \sum_{k \in \cd{l}} \sum_{j \in \ZZ^d} \tilde\omega_j \lvert u(k-j) \rvert .
\end{align*}
Choose $R = R(u) \in \NN$ such that 
\[
  \hat C \euler^{-\alpha R/2} \left[\frac{16}{\lVert u \rVert_1} + \frac{\overline{u}}{4\lVert u \rVert_1^2}  \right] \leq \frac{1}{8},
\]
where $\hat C = \hat C (C , \alpha , d)$ is the constant from Lemma~\ref{lemma:eig_perturb}. Later it will be convenient that $l \geq l_9^* = l_9^* (u) \colonequals 2R$.
We split the second sum in $j \in \cd{l+R}$ and $j \not \in \cd{l+R}$ and obtain by Lemma~\ref{lemma:eig_perturb} that
\begin{equation} \label{eq:upper2}
 \Bigl\langle \psi , \bigl(h_{\tilde \omega}^{\cd{l} , \rm N} \bigr)^2 \psi \Bigr\rangle \leq \frac{1}{\lvert \cd{l} \rvert} \frac{8l^{-2}}{\beta} \left(\lVert u \rVert_1 \sum_{j \in \cd{l+R}} \tilde\omega_j + \lvert \cd{l} \rvert \frac{8 l^{-2}}{\beta \lVert u \rVert_1} \hat C \euler^{-\alpha R /2} \right).
\end{equation}
Here $\hat C$ is again the constant from Lemma~\ref{lemma:eig_perturb}. We also need a lower bound for $\langle \psi , h_{\tilde \omega}^{\cd{l} , \rm N} \psi \rangle$. We have
\begin{align*}\label{eq:lower_bound}
  \bigl\langle \psi , h_{\tilde \omega}^{\cd{l} , \rm N} \psi \bigr\rangle 
 &= \frac{1}{\lvert \cd{l} \rvert} \sum_{k \in \cd{l}} \Biggl( \sum_{i \in \cd{l+R} }\tilde\omega_i u(k-i) + \sum_{i \not\in \cd{l+R}}\tilde\omega_i u(k-i) \Biggr) \\
 &= \frac{1}{\lvert \cd{l} \rvert}\Biggl(\overline u \!\!\! \sum_{i \in \cd{l+R}} \!\!\! \tilde \omega_i - \!\!\!\!\! \sum_{i \in \cd{l+R}} \sum_{k \not\in \cd{l}} \tilde\omega_i u(k-i) +  \sum_{k \in \cd{l}} \sum_{i \not\in \cd{l+R}} \tilde\omega_i u(k-i) \Biggr) .
\end{align*}
For the third summand we have by Lemma~\ref{lemma:eig_perturb}
\[
 \frac{1}{\lvert \cd{l} \rvert}\sum_{k \in \cd{l}} \sum_{i \not\in \cd{l+R}} \tilde\omega_i u(k-i) \geq   -\frac{1}{\lvert \cd{l} \rvert}\sum_{k \in \cd{l}} \sum_{i \not\in \cd{l+R}} \tilde\omega_i \lvert u(k-i) \rvert \geq -  \frac{8l^{-2}}{\lVert u \rVert_1 \beta} \hat C \euler^{-\alpha R / 2} .
\]
For the second sum we have using (a version with $u(k-x)$ instead of $u(x-k)$ of) Lemma~\ref{lemma:eig_perturb} that for $l \geq l_9^*$
\begin{align*}
  \sum_{i \in \cd{l+R}} \sum_{k \not\in \cd{l}} \tilde\omega_i u(k-i) &= 
 \sum_{i \in \cd{l-R}} \sum_{k \not\in \cd{l}} \tilde\omega_i u(k-i) + \sum_{i \in \cd{l+R} \setminus \cd{l-R}} \sum_{k \not\in \cd{l}} \tilde\omega_i u(k-i) \\
& \leq \lvert \cd{l-R} \rvert \frac{8l^{-2}}{\beta \lVert u \rVert_1} \hat C \euler^{-\alpha R / 2} + \frac{8 l^{-2}}{\beta \lVert u \rVert_1} \lvert \cd{l+R} \setminus \cd{l-R} \rvert \lVert u \rVert_1 \\
& \leq \lvert \cd{l-R} \rvert \frac{8l^{-2}}{\beta \lVert u \rVert_1} \hat C \euler^{-\alpha R / 2} + \frac{8 l^{-2}}{\beta} 4dR(l+R)^{d-1} .
\end{align*}
If $l$ is sufficiently large, i.e.\ $l \geq l_{10}^*$ for some $l_{10}^* = l_{10}^* (R,d) = l_{10}^* (u)$, we have
\[
 \frac{1}{\lvert \cd{l} \rvert}\sum_{i \in \cd{l+R}} \sum_{k \not\in \cd{l}} \tilde\omega_i u(k-i) \leq \frac{8l^{-2}}{\beta \lVert u \rVert_1} \hat C \euler^{-\alpha R / 2} + \frac{l^{-2}}{8 \beta} .
\]
Putting everything together we arrive at
\begin{equation} \label{eq:lower1}
 \bigl\langle \psi , h_{\tilde \omega}^{\cd{l} , \rm N} \psi \bigr\rangle  \geq
\frac{\overline u}{\lvert \cd{l} \rvert}  \sum_{i \in \cd{l+R}} \!\!\! \tilde \omega_i 
- 2 \frac{8l^{-2}}{\beta \lVert u \rVert_1}\hat C \euler^{-\alpha R / 2} - \frac{l^{-2}}{8 \beta} .
\end{equation}
We apply Temple's inequality and obtain by Ineq.~\eqref{eq:lower1}, Ineq.~\eqref{eq:upper2} and Ineq.~\eqref{eq:lower2} that 
\begin{align*}
 \lambda_1 (h_\omega^{\cd{l} , \rm N}) &\geq  \lambda_1 (h_{\tilde\omega}^{\cd{l} , \rm N}) - \frac{l^{-2}}{8 \beta} 
\geq \bigl\langle \psi , h_{\tilde \omega}^{\cd{l} , \rm N} \psi \bigr\rangle - \frac{\Bigl\langle \psi , \bigl(h_{\tilde \omega}^{\cd{l} , \rm N}\bigr)^2 \psi \Bigr\rangle}{\xi  - \bigl\langle \psi , h_{\tilde \omega}^{\cd{l} , \rm N} \psi \bigr\rangle } - \frac{l^{-2}}{8 \beta} \\
& \geq \frac{3}{4} \frac{\overline{u}}{\lvert \cd{l} \rvert} \sum_{j \in \cd{l+R}} \tilde\omega_j - \frac{l^{-2}}{4\beta} - \frac{l^{-2}}{\beta} \hat C \euler^{-\alpha R/2} \Bigl[\frac{16}{\lVert u \rVert_1} + \frac{\overline{u}}{4 \lVert u \rVert_1^2}  \Bigr] .
\end{align*}
Set $l_8^* = \max\{l_9^* , l_{10}^*\}$. By our choice of $R$ we finally obtain for $l \geq l_8^*$ that
\[
\lambda_1 (h_\omega^{\cd{l} , \rm N}) \geq \frac{3}{4} \frac{\overline{u}}{\lvert \cd{l} \rvert} \sum_{j \in \cd{l+R}} \tilde\omega_j - \frac{3}{8} \frac{l^{-2}}{\beta} .
\]
Assume that the statement of the proposition is wrong. Then
\[
\biggl\lvert \biggl\{k \in \cd{l} \colon \omega_k \geq \frac{4l^{-2}}{\beta \overline{u}} \biggr\} \biggr\rvert > \frac{11}{12} \frac{\lvert \cd{l}\rvert}{2} .
\]
Since $\omega_k \geq 4l^{-2} / (\beta \overline{u})$ implies $\tilde\omega_k \geq 4l^{-2} / (\beta \overline{u})$ we have
\[
\frac{l^{-2}}{\beta} > \lambda_1 (h_\omega^{\cd{l} , \rm N}) \geq \frac{3}{4} \frac{\overline{u}}{\lvert \cd{l} \rvert} \sum_{j \in \cd{l+R}} \tilde\omega_j - \frac{3}{8} \frac{l^{-2}}{\beta} \geq \frac{l^{-2}}{\beta} .
\]
This is a contradiction. 
\end{proof}
\begin{proposition} \label{prop:lifshitz}
Let Assumption~\ref{ass:exp} be satisfied, $\overline{u}\colonequals \sum_{k \in \ZZ^d} u(k) > 0$, $\zeta \in (0,2)$, $\xi > 0$, $\beta_0$ as in Proposition~\ref{prop:first} and assume that there is $\epsilon_0 > 0$ with
\[
\PP (\omega_0 < \epsilon_0) \leq \frac{1}{12}. 
\]
Then there exists and $l_{11}^* = l_{11}^* (u,\mu,\beta_0,\zeta,\xi) < \infty$, such that if we pick $l \geq l_{11}^*$ satisfying
\begin{equation} \label{eq:ungerade}
 \frac{\lfloor 2l+1 \rfloor}{\lfloor 2l^{1-\zeta / 2} \beta_0^{-1/2}+1 \rfloor} \in 2 \NN + 1 ,
\end{equation}
and assume that Assumption~\ref{ass:small_neg} is satisfied for $\delta = l^{\zeta - 2} / (8 \omega_+)$, then we have 
 \[
  \PP \bigl( \lambda_1 (h_\omega^l) < l^{-2+\zeta} \bigr) \leq l^{-\xi}  .
 \]
\end{proposition} 
\begin{remark} \label{remark:unbounded}
 The set 
 \[
 \left \{l > 0 \colon \frac{\lfloor 2l+1 \rfloor}{\lfloor 2l^{1-\zeta / 2} \beta_0^{-1/2}+1 \rfloor} \in 2 \NN + 1 \right\}
 \]
is an unbounded set. This follows readily from the following
fact: If $(a_n)$ and $(b_n)$
are sequences of natural numbers satisfying $a_n \le a_{n+1} \le a_n + 1$
and $b_n \le b_{n+1} \le b_n + 1$ for all $n$ and $a_n / b_n \to \infty$,
then every natural number larger than $a_1 / b_1$ can be realized by a quotient $a_n
/ b_n$.
\end{remark}

\begin{proof}[Proof of Proposition~\ref{prop:lifshitz}]
We set $\tilde l \colonequals l^{1-\zeta / 2} \beta_0^{-1/2}$ and assume that $l$ is large enough, say $l \geq l_{12}^* = l_{12}^* (u , \beta_0 , \zeta)$, such that $\tilde l \geq l_8^*$ with $l_8^*$ from Proposition~\ref{prop:first}. By construction we have $n\colonequals \lfloor 2l+1 \rfloor / \lfloor 2 \tilde l+1 \rfloor \in 2\NN + 1$. Hence we can divide the cube $\cd{l}$ into smaller disjoint cubes $\cd{}^j = \cd{\tilde l} (z_j)$ with appropriate centers $z_j$, $j$ from one to $n^d$, and the property that $\cd{l} = \mathbin{\dot{\cup}_j} \cd{}^j$. Following \cite{Veselic-01}, see also \cite{Kirsch-08} for the discrete setting, we arrive at
 \begin{align} \label{eq:Neumann}
  \PP \bigl( \lambda_1 (h_\omega^l) < l^{-2+\zeta} \bigr) 
&= \sum_{j=1}^{n^d} \PP \bigl( \lambda_1 (h_\omega^{\cd{}^j , \rm N}) <  \beta_0^{-1} \tilde{l}^{-2} \bigr) .
 \end{align}
By translation invariance it remains to estimate $P\colonequals\PP ( \lambda_1 (h_\omega^{\cd{\tilde l} , \rm N}) < \beta_0^{-1} \tilde{l}^{-2} )$. Since $\tilde l \geq l_8^*$ and Assumption~\ref{ass:small_neg} is satisfied for $\delta = \tilde{l}^{-2}/(8 \beta_0 \omega_+)$ we can apply Proposition~\ref{prop:first} and obtain
\begin{align*}
P & \leq \PP \Biggl( \biggl\lvert \biggl\{k \in \cd{\tilde l} \colon \omega_k < \frac{4 \tilde{l}^{-2}}{\beta_0 \overline{u}} \biggr\} \biggr\rvert > \frac{13}{12} \frac{\lvert \cd{\tilde l} \rvert}{2} \Biggr) \\ &= \PP \Biggl( \frac{1}{\cd{\tilde l}} \sum_{k \in \cd{\tilde l}} \Eins_{\bigl[0 , \frac{4 \tilde{l}^{-2}}{\beta_0 \overline{u}} \bigr)} (\omega_k) \geq p + \Bigl(\frac{13}{24} - p\Bigr) \Biggr), \quad p \colonequals \PP \biggl(\omega_k \in \bigl[0 , \frac{4\tilde{l}^{-2}}{\beta_0 \overline{u}} \bigr) \biggr) .
\end{align*}
If $l \geq l_{13}^* = l_{13}^* (\mu , \beta_0,\zeta , u)$ then $p \leq 1/12$. Hence $13/24 - p$ is positive and Bernstein's inequality gives
\begin{equation} \label{eq:Bernstein}
 P \leq \euler^{- 2 \lvert \cd{\tilde l} \rvert \bigl( \frac{13}{24} - p \bigr)} \leq \euler^{-\lvert \cd{\tilde l} \rvert \bigl( \frac{11}{12} \bigr)^2} . 
\end{equation}
From Eq.~\eqref{eq:Neumann} and Ineq.~\eqref{eq:Bernstein} we infer 
\[
 \PP \bigl( \lambda_1 (h_\omega^l) <  l^{-2+\zeta} \bigr) \leq n^d \euler^{-\lvert \cd{\tilde l} \rvert \bigl( \frac{11}{12} \bigr)^2} 
 \leq \lfloor 2l+1 \rfloor^d \euler^{-( l^{- \zeta/2} \beta_0^{-1/2})^d} .
\]
The result follows by choosing $l$ sufficiently large, depending only on $u$, $\mu$, $\beta_0$, $\zeta$ and $\xi$.
\end{proof}
\begin{proposition} \label{pro:initial_small}
Let Assumption~\ref{ass:exp} be satisfied, $\overline{u} > 0$ and assume that there is $\epsilon_0 > 0$, such that
\[
 \PP (\omega_0 < \epsilon_0) \leq \frac{1}{12} .
\]
Let further $\xi$, $\kappa$, $\beta$, $q$ and $m_0$ be as  required in Definition~\ref{ass:ini}, $\zeta \in (2 - 2(1 - \beta) / \kappa , 2)$, 
and let $\beta_0$ be as in Proposition~\ref{prop:first}. 
\par
Then there exists $\delta > 0$ and $\epsilon_{l_0} > 0$, both depending only on $u$, $\mu$, $\beta_0$, $\zeta$, $\xi$, $\kappa$, $\beta$ and $q$, 
such that if Assumption~\ref{ass:small_neg} is satisfied for $\delta$, then the initial scale estimate holds in in the interval $[ - \epsilon_{l_0}/2 , \epsilon_{l_0}/2]$.
\end{proposition}
The length $\epsilon_{l_0}$ of the localization interval $[ - \epsilon_{l_0}/2 , \epsilon_{l_0}/2]$
can be determined following formulas \eqref{eq:interval_size},  \eqref{eq:loc_interval}, and \eqref{eq:l0}.
Similarly, $\delta$ equals $l_0^{\zeta - 2} / (8\omega_+)$ for $l_0$ as in  \eqref{eq:l0}.
\begin{proof}
Note that we know from Remark~\ref{remark:unbounded} that for each constant $K$ there is an $l \geq K$ which satisfies condition \eqref{eq:ungerade}. For all $l \geq l_{11}^*$ satisfying condition \eqref{eq:ungerade} and assuming that Assumption~\ref{ass:small_neg} is satisfied for $\delta = l^{\zeta - 2} / (8 \omega_+)$, we have by Proposition~\ref{prop:lifshitz} 
   \[
  \PP \bigl( \lambda_1 (h_\omega^{l}) \geq l^{-2+\zeta} \bigr) \geq1- l^{-\xi}  .
 \]
 Let $\Omega_{\rm L} \colonequals \{\omega \in \Omega \colon \lambda_1 (h_\omega^{l}) \geq l^{-2+\zeta}\}$. From Ineq.~\eqref{eq:perturb} we infer that the set
 \begin{multline*}
  \Omega_{\rm L}^* \colonequals \{ \omega \in \Omega \colon \lambda_1 (h_{\omega'}^{l}) \geq l^{-2+\zeta} - C_{u,\omega_+} \euler^{-3 l \alpha / 2}\\ \text{for all $\omega' \in \Omega$ with $\Pi_{\cd{4l}} \omega' = \Pi_{\cd{4l}} \omega$} \} .
 \end{multline*}
satisfies $\Omega_{\rm L}^* \supset \Omega_{\rm L}$, hence $\PP (\Omega_{\rm L}^*) \geq 1 - l^{-\xi}$. 
We assume that $l \geq l_{14}^* = l_{14}^* (\zeta , u , \omega_+)$ such that 
\begin{equation} \label{eq:interval_size}
\epsilon_l \colonequals  l^{-2 + \zeta} - C_{u,\omega_+}\euler^{-3 l \alpha / 2}> 0.
 \end{equation}
Let $\omega \in \Omega_{\rm L}^*$ and 
\begin{equation} \label{eq:loc_interval}
 E \in I_l = \biggl[- \frac{\epsilon_l}{2} , \frac{\epsilon_l}{2}\biggr] .
\end{equation}
Then $d (E , \sigma (h_{\omega'}^l)) \geq \epsilon_l / 2$ for all $\omega' \in \Omega$ with $\Pi_{\cd{4l}} \omega' = \Pi_{\cd{4l}} \omega$, or with our shorthand notation, $\tilde d (E , \sigma (h_{\omega}^l)) \geq \epsilon_l / 2$. The Combes-Thomas estimate, see e.g.\ \cite{Klopp-02}, gives that there is a universal constant $C$ such that for all $n,m \in \cd{l}$ and all $\omega' \in \Omega$ with $\Pi_{\cd{4l}} \omega' = \Pi_{\cd{4l}} \omega$ that
\[
 \lvert G_{\omega'}^{\cd{l}} (E , n,m) \rvert \leq \left( \frac{C}{\rho (E)} \right)^2 \euler^{- \rho (E) \lVert m-n \rVert / C}
\]
where
\[
 \rho (E) = \inf \Bigl\{ \sqrt{d (E, \sigma (h_{\omega'}^l))} , 1/4 \Bigr\} .
\]
Hence, for each $\epsilon \in (0,1)$ we find $l_{15}^* = l_{15}^* (\epsilon , u,\mu,\zeta)$ such that for all $l \geq l_{15}^*$ we have
\[
 \sup_{w \in \partial^{\rm i} \cd{l}}\lvert G_{\omega'}^{\cd{l}} (E; 0,w) \rvert \leq \euler^{-(1-\epsilon) l^{\zeta / 2 - 1} l} 
\]
for all $\omega' \in \Omega$ with $\Pi_{\cd{4l}} \omega' = \Pi_{\cd{4l}} \omega$. Hence, we have for all $\epsilon \in (0,1)$ and all $l \geq \max\{l_{11}^* , l_{14}^* , l_{15}^*\}$ satisfying Eq.~\eqref{eq:ungerade} (assuming for the moment that Assumption~\ref{ass:small_neg} is satisfied for an appropriate $\delta$)
\[
 \PP (\forall E \in I_l \ \text{the cube} \ \cd{l} \ \text{is uniformly $((1-\epsilon) l ^{\zeta / 2 - 1} , E)$-regular}) \geq 1 - l^{-\xi} .
\]
Let $z_1 , z_2 \in \ZZ^d$ with $\lVert z_1 - z_2 \rVert_\infty > 8l$ and set $P \colonequals \PP (B_{l} (z_1 , z_2 , m_l , I_l))$ where $m_l = (1-\epsilon) l ^{\zeta / 2 - 1}$. We use translation invariance and the independence of two ``disjoint'' cylindersets and obtain
\begin{align*}
 P &\leq \PP \Bigl( \{\exists E \in I_l \colon \text{$\cd{l} (z_1)$ is not uniformly $(m_l , E$)-regular}\} \Bigr. \\ &\qquad \Bigl. \cap \{\exists E \in I_l \colon \text{$\cd{l} (z_2)$ is not uniformly $(m_l , E$)-regular}\}  \Bigr) \\
 &=\PP \Bigl( \{\exists E \in I_l \colon \text{$\cd{l}$ is not uniformly $(m_l , E$)-regular}\} \Bigr)^2 \leq l^{-2\xi}
\end{align*}
Since $\zeta > 2\beta$, which follows from our assumption $\zeta > 2 - 2(1 - \beta) / \kappa$, there is $l_{16}^* = l_{16}^* (\zeta , \beta , \epsilon)$ such that for $l \geq l_{16}^*$ we have $m_l > l^{\beta - 1}$ as required in Definition~\ref{ass:ini}. 
If we pick 
\begin{equation} \label{eq:l0}
l_0 > \max \left\{ \left( \frac{3}{(1-q)(1-\epsilon)} \right)^{\frac{2}{\zeta - ( 2 - 2(1-\beta)/\kappa )}} , l_{11}^* , l_{14}^* , l_{15}^* , l_{16}^* , l^* \right\}
\end{equation}
 satisfying Eq.~\eqref{eq:ungerade}, set $m_0 = m_{l_0}$, and assume that Assumption~\ref{ass:small_neg} is satisfied for $\delta = l_0^{\zeta - 2} / (8\omega_+)$, then the initial scale estimate is satisfied in $I_{l_0}$. Since $\zeta > 2 - 2(\beta - 1) / \kappa$, the first condition in \eqref{eq:l0} ensures that $l_0 > \overline{l}$ as required in Definition~\ref{ass:ini}.
\end{proof}
\begin{proof}[Proof of Theorem~\ref{thm:loc:large}] The statement follows from Lemma~\ref{lemma:ini_large} and Theorem~\ref{thm:loc_under_ini}
 \end{proof}
\begin{proof}[Proof of Theorem~\ref{thm:loc:weak}] The statement follows from Proposition~\ref{pro:initial_small} and Theorem~\ref{thm:loc_under_ini}.
\end{proof}
%
%% %%%%%%%%%%%%%%%%%%%%%%%%%%%%%%%%%%%%%%%%%%%%%%%%%%%%%%%%%%%%%%%%%%%%%%%%%%%%%%%%%%%
%--------------------------------------------------------------------------------
%
%          MSA Cont
% %--------------------------------------------------------------------------------
%%
%%%%%%%%%%%%%%%%%%%%%%%%%%%%%%%%%%%%%%%%%%%%%%%%%%%%%%%%%%%%%%%%%%%%%%%%%%%%%%%%%
%%
%

\section{Localization via multiscale analysis (continuous model)} \label{sec:loc_cont}
In Section~\ref{sec:loc_discrete} we apply the multiscale analysis \`a la \cite{KirschSS-98b} to the discrete alloy-type model with exponentially decaying single-site potential. Beyond doubt, on the basis of the Wegner estimate from Theorem~\ref{theorem:wegner_c} one could do the same for the (continuous) alloy-type model with exponentially decaying convolution vector. 
\par
However, to keep things short, we just note that Proposition~\ref{prop:unif_wegner_type_cont} replaces \cite[Lemma~3.4]{KirschSS-98b}, which is sufficient for the induction step of the multiscale analysis. Hence, once an appropriate initial length scale estimate is satisfied one obtains localization. More precisely, Proposition~\ref{prop:unif_wegner_type_cont} and the multiscale analysis \`a la \cite{KirschSS-98b} imply the following theorem.
\begin{theorem}[localization, continuous model]
Assume that $U$ is a generalized step function and assume that Assumptions~\ref{ass:bv} and \ref{ass:exp} are satisfied. Denote by $a$ the infimum of the almost sure spectrum of $H_\omega$ and assume further that for any $\xi > 0$ and $\beta_0 \in (0,2)$ there is an $l^* = l^* (\xi , \beta_0)$ such that 
\begin{equation} \label{eq:initialKSS}
 \PP \Bigl( \tilde d \bigl(a , \sigma (H_\omega^l)  \bigr) \leq l^{\beta_0 - 2}  \Bigr) \leq l^{-\xi} .
\end{equation}
Then, for almost every $\omega \in \Omega$, the spectrum of $H_\omega$ is only of pure point type in a neighborhood of $a$ with exponentially decaying eigenfunctions.
\end{theorem}
Here, the uniform distance $\tilde d$ from Ineq.~\eqref{eq:initialKSS} is defined by
\[
 \tilde d \bigl(a , \sigma (H_\omega^l)  \bigr) =\inf_{\omega_1^{\bot} \in \, \Omega_{\ZZ^d\setminus \cd{4l}}} 
d\Bigl(E,\sigma\bigl(H^{l}_{(\omega_1 , \omega_1^\bot )}\bigr)\Bigr) 
\]
where $\omega_1 \in \Omega_{\cd{4l}}$ is defined by 
$\omega_1=\Pi_{\cd{4l}}\omega$.
\begin{remark}
Let us finally discuss the validity of an initial scale estimate as formulated in Ineq.~\eqref{eq:initialKSS}. 
If the single-site potential $U$ is non-negative and satisfies $\lvert U(x) \rvert \leq C \lVert x \rVert^{-m}$ for $m$ large, Ineq.~\eqref{eq:initialKSS} is a well known fact, see e.g.\ \cite{KirschSS-98a,KirschSS-98b} for the case where $U$ is compactly supported. If the single-site potential changes its sign and has unbounded support, far less is known. 
However, similarly to Lemma~\ref{prop:lifshitz} of Section~\ref{sec:ini} one can prove Ineq.~\eqref{eq:initialKSS} 
for the alloy-type model on $L^2 (\RR^d)$ if the single-site potential $U$ is a generalized step function with an exponentially decaying convolution vector of a small negative part.
In the case that the single-site potential is even compactly supported this has been done in Section~5
of \cite{Veselic-01}.
\end{remark}

\subsubsection*{Acknowledgement}
Financial support of the Deutsche Forschungsgemeinschaft is gratefully 
acknowledged. We thank the referee for useful remarks.
%
% \bibliographystyle{alpha}
% \bibliographystyle{abbrv}
% \bibliography{LPTV-lit}
%

\end{document}